\documentclass[12pt]{article}

\usepackage{lmodern}
\providecommand{\keywords}[1]{\textbf{\small{Key words:}} #1}
\providecommand{\AMS}[1]{\textbf{\small{AMS subject classifications:}} #1}
\usepackage[utf8]{inputenc} 
\usepackage[T1]{fontenc}    
\usepackage{hyperref}       
\usepackage{url}            
\usepackage[english]{babel}
\usepackage{amsfonts, amsmath, bm, amssymb}       
\usepackage{graphicx}
\usepackage{siunitx} 

\usepackage[width=16cm, top=2.5cm]{geometry}
\graphicspath{{Figures/}}
\usepackage{amsthm} 
\usepackage{float} 
\usepackage{physics} 

\usepackage{mathptmx}
\usepackage[font=small, labelfont=bf]{caption}

\usepackage{xcolor}
\definecolor{colorBlack}{rgb}{0.,0.,0.}
\def\black{\color{colorBlack}}

\newcommand{\calO}{\mathcal{O}}

\newcommand{\calR}{\mathcal{R}}

\newcommand{\R}{\mathbb{R}}
\newcommand{\C}{\mathbb{C}}

\newcommand{\rmi}{\mathrm{i}}

\newcommand{\eps}{\varepsilon}

\newcommand{\w}{r}
\newcommand{\F}{\mathbf{F}}
\newcommand{\M}{\mathbf{M}}
\def\eps{\varepsilon}
\def\L{L_{pp}}
\newcommand{\bT}{\textbf{T}}
\newcommand{\ba}{\textbf{a}}
\newcommand{\bb}{\textbf{b}}
\newcommand{\bs}{\textbf{s}}
\newcommand{\bv}{\textbf{v}}
\newcommand{\T}{Z}
\newcommand{\RR}{\Lambda}
\def\SS{\Gamma}

\newcommand{\tu}{\tilde{u}}
\def\tt{\tilde{\tau}}
\newcommand{\tK}{\widetilde{K}}

\newcommand{\tC}{\widetilde{C}}

\newcommand{\teps}{\tilde{\eps}_r}
\newcommand{\epsrd}{\underline{\eps}_\w}
\newcommand{\epspsid}{\underline{\eps}_\psi}
\newcommand{\XT}{X_P}

\newcommand{\xTs}{x_{T_s}}

\newcommand{\barm}{\bar{m}}
\newcommand{\mL}{m}
\newcommand{\cA}{k_{20}}
\newcommand{\cB}{k_{21}}
\newcommand{\ctA}{k_{10}}
\newcommand{\ctB}{k_{11}}
\newcommand{\cC}{k_{02}}
\newcommand{\cZ}{k_{01}}
\newcommand{\rp}{\rho} 
\newcommand{\ampc}{b_r} 
\newcommand{\kk}{z}
\newcommand{\target}{*}

\DeclareMathOperator{\sgn}{sgn}

\theoremstyle{plain}
\newtheorem{theorem}{Theorem}[section] 

\newtheorem{remark}[theorem]{Remark} 

\numberwithin{equation}{section} 

\title{Bifurcation control for a ship maneuvering model with nonsmooth nonlinearities}

\date{June 26, 2022}

\author{Miriam Steinherr Zazo\thanks{University of Bremen, Germany, Department 3 -- Mathematics, m.steinherr@uni-bremen.de}  \and Jens D.~M.~Rademacher\thanks{University of Bremen, Germany, Department 3 -- Mathematics, jdmr@uni-bremen.de}}

\hypersetup{
	pdftitle={Lyapunov coefficients for Hopf bifurcations in systems with piecewise smooth nonlinearity},
	pdfsubject={},
	pdfauthor={Miriam Steinherr Zazo, Jens D.~M.~Rademacher},
	pdfkeywords={First keyword, Second keyword, More},
}

\begin{document}
\maketitle

\begin{abstract}
	We consider a widely used form of models for ship maneuvering, whose nonlinearities entail continuous but nonsmooth second-order modulus terms. For such models bifurcations of straight motion are not amenable to standard center manifold reduction and normal forms. Based on a recently developed analytical approach, we nevertheless determine the character of local bifurcations when stabilizing the straight motion course with standard proportional control. For a specific model class we perform a detailed analysis of the linearization to determine the location of these bifurcations in the control parameter space and its dependence on selected design parameters. By computing the analytically derived characteristic parameters, we find that `safe' supercritical Andronov--Hopf bifurcations are typical. Through numerical continuation we provide a more global bifurcation analysis, which identifies the arrangement and relative location of stable and unstable equilibria and periodic orbits.
\end{abstract}


\medskip
\AMS{{\small 34H15, 34H20, 37N35, 93D20}}

\section{Introduction}
{\black A typical task in ship maneuvering is to maintain a specified heading. In standard differential equation models for ship maneuvering such straight motion appears as an equilibrium point. We consider here the frequent situation in which this is unstable in absence of control, and study the effectiveness of a standard proportional control to stabilize it. 
	We do not specifically design the control for this task, which would be a broader goal in bifurcation control \cite{ABED1986,CMW2000,Kang2003Controllability}. Instead, we identify the possibility to stabilize the straight motion in terms of the given control gains. 
	Our main interest lies in the resulting nonlinear effects. The given control combines the steering angle $\eta$ with the yaw angle $\psi$, where $\psi=0$ is the desired heading, and the yaw velocity $r=\dot \psi$. The standard proportional yaw damping and yaw restoring control takes the form $\eta=\eps_\w r+ \eps_\psi \psi$ with control gains $\eps_\w$ and $\eps_\psi$ \cite{PAPOULIAS1994,STC2007,Tigkas2019}. Stabilization on a linear level identifies gain margins at which linear growth rates of perturbations switch from positive to negative \cite{PAPOULIAS1994}. These gain margins form a stability boundary curve in the control parameter plane.
	
	It turns out that for the selected model class the crossing from the unstable to the stable region is either a pitchfork or an Andronov--Hopf bifurcation. It is well known that these bifurcations are generically either subcritical or supercritical \cite{Carr}. 
	For a subcritical bifurcation, the linearly stable state co-exists with a nearby unstable bifurcated state, i.e., the effect of bifurcation is inside the stable region in control parameter space. Although linear analysis predicts stability, the basin of attraction can be very small. When random perturbations exceed this small basin, local information is insufficient to determine how the system further evolves. 
	For a supercritical bifurcation, the unstable state co-exists with a stable bifurcated state before parameters have reached the stable region. In this case the linear stability boundary is a safe estimate for stabilization. It is therefore important to be able to determine whether the bifurcation in a control scenario is supercritical or subcritical \cite{PAPOULIAS1994}.
	
	The criticality can be inferred from the so-called first Lyapunov coefficient for the Hopf bifurcation or its analogue for the pitchfork bifurcation. 
	If the nonlinear terms are smooth, then this coefficient can be computed by means of a center manifold reduction and normal form analysis \cite{Carr}. 
	We refer to \cite{ABED1986,CMW2000,GS19} and the references therein for studies in the context of smooth ship models and their relation to bifurcation control. We also mention that various studies have numerically investigated bifurcations and the resulting branches of solutions for smooth models of ship maneuvering without focus on the Lyapunov coefficient  \cite{SB2014,Spy2017,SpyrouThompson2000,Tigkas2019}. 
	However, the standard class of models for ship maneuvering that we consider features nonlinearities with continuous nonsmooth terms \cite{Abkowitz1969,StandShips2006,InitialPaper,FossenHandbook,STC2007,ToxopeusPaper,ToxopeusThesis}. Here the smooth theory is not applicable in general. 
	In \cite{STC2007}, a purely numerical study of bifurcations for such a model is presented, which avoids the problem and does not determine Lyapunov coefficients. 
	
	Originally motivated by these ship models, we have developed a theoretical framework that admits to derive the first Lyapunov coefficient for a broad class of continuous nonsmooth models \cite{SteinMacher2020}. The main purpose of the current paper is to illustrate the application of this theory to models of ship maneuvering. As a preparatory part for this nonlinear analysis, but of independent interest, we provide a detailed analytical study of stability boundaries for the selected model class. This provides explicit formulas for stability boundaries in terms of the control gains and ship design parameters. In particular, the location of the propulsion force on the hull has an important impact on the geometry and location of the stability boundary. 
	Concerning criticality of bifurcations, we find that in the considered models all are supercritical and thus safe.  
	In order to gain additional insight into the arrangement of the bifurcating periodic orbits more globally in control parameter space, we perform numerical continuation studies. Some of the bifurcating solutions have broadly varying yaw angle, so that for their study we need to account for the global cylindrical topology of phase space. We find that the continuation of these periodic solutions and those from the Hopf bifurcation terminate in heteroclinic bifurcations that involve the cylindrical geometry. 
	
	\medskip
	This paper is structured as follows. 
	We present a brief background regarding marine craft hydrodynamics in \S\ref{Marine_Craft_Hydro}. Here we also present the specific equations of motion for the ship model considered in this paper and detail the kind of control used.
	In the main section \S\ref{Theoretical_Analysis}, we theoretically investigate the stability of the straight line motion and bifurcations from it. Included in this  analysis is the impact of some ship design parameters.
	The numerical bifurcation and continuation analysis is presented in \S\ref{Numerical_Bif_An}, illustrating the preceding theoretical study. For the continuation of solutions that have widely varying yaw angle, the control law is modified to depend periodically on the yaw angle. We show selected periodic orbits in phase space as well as the corresponding ship tracks in Earth-fixed coordinates.
	Finally, in \S\ref{Discussion} we discuss the results and present possible directions for further research.
}

\section{Model Equations and Background}
\label{Marine_Craft_Hydro}

Here we briefly discuss ingredients that are most relevant for our subsequent analysis. The kinetics is based on Newton's second law and Euler's axiom{\black . With these, the rigid-body equations of motion take the form $\M\dot{\bv} = \F(\bv)$, where $\bv$ is the vector of ship-fixed velocities and $\dot{\bv}$ its time derivative. The matrix $\M$ contains the mass coefficients, which include the added mass due to the water displacement and moments of inertia. The vector $\F$ contains the} forces from the hull, rudder,  propeller and hydrodynamics, as well as the Coriolis term. {\black Detailed derivations of the equations of motion for a marine craft can be found in, e.g., \cite{FossenHandbook}.}

In the modeling of the hydrodynamic forces, {\black we are interested in} the nature of the arising nonlinear terms since their nonsmooth character is decisive in the analysis of bifurcations. 
One approach to the nonlinear terms follows the drag equation for high Reynolds number given by 
$$F_D = -\frac{1}{2}\rho C_D A u\abs{u},$$
where $u$ is the velocity of the body, $\rho$ the density of the water, $C_D$ the drag coefficient, and $A$ the effective drag area, \cite{FossenHandbook}. This equation is a consequence of {\black (I)} the experimental observation that the drag force $F_D=F_D(\rho,A,u)$ is a function of $\rho,A,u${; \black (II)} the fact that{\black,} as an opposing force, {\black $F_D$} must be odd with respect to $u$; and {\black (III)} that dimensional analysis in a power law ansatz $F_D=\rho^\alpha A^\beta u^\gamma C$ for $u>0$ with exponents $\alpha,\beta,\gamma\in\R$ and a constant $C$, implies $\alpha=\beta=1$, $\gamma=2$. Indeed, polynomial regression studies {\black on the representation} of hydrodynamic forces confirm that $u|u|$, $v|v|$, $r|r|$ and the mixed terms $v|r|$, $r|v|$, are the relevant higher-order terms \cite{OntheOrder,Wolff81}. These \textit{second-order modulus terms} can be regarded as square law damping in this context. 
{\black In addition, these nonlinear terms can be motivated by a Taylor expansion to second-order, where the absolute value is used to correct the signs \cite{StandShips2006}.} 
Ship models with third-order Taylor approximations {\black are also used in the literature}, e.g., {\black \cite{PAPOULIAS1995,Spy2017,Tigkas2019}}. 
However, when it comes to bifurcation analysis, there is a significant difference between the second-order modulus and cubic terms as discussed in detail in \cite{SteinMacher2020} {\black and reflected in our bifurcation analysis of} \S\ref{NonLinearANalysis}.

The specific model equations that we will investigate are a variation of the $3$ {\black degree-of-freedom} model from \cite{StandShips2006,ToxopeusThesis}, for which some basic analysis was conducted in \cite{InitialPaper}. The model parameters stem from the `Hamburg Test Case' (HTC) characteristics,  
{\black that we collect in Appendix \ref{HTC_param} as needed.} 
We adopt these values throughout, except when analyzing the impact of selected parameter changes.
The general {\black dimensional} $3$ {\black degree-of-freedom} model takes the form
\begin{equation}
	\begin{pmatrix}
		\barm+\barm_{uu} 	& 0 		& 0 \\
		0 			& \barm+\barm_{vv}	& \barm_{v\w} \\
		0 			& \barm_{\w v}	&\bar{I}_z+ \barm_{\w\w}
	\end{pmatrix}
	\begin{pmatrix}
		\dot{u}\\
		\dot{v}\\
		\dot{\w}\\
	\end{pmatrix}=
	\begin{pmatrix}
		\barm v\w+X\\
		-\barm{u}\w+Y\\
		N
	\end{pmatrix},
	\label{OriginalSys}
\end{equation}
where $u,v$ and $\w$ are the surge, sway and yaw velocities, respectively. The external forces $X,Y,N$ for the `rudder model' of \cite{ToxopeusThesis} are of the form $X=X_H+X_R+X_P$, $Y=Y_H+Y_R$, $N=N_H+N_R$ with the contributions from the hull (H), the rudder (R) and the propeller (P) of the vessel. 

In order to facilitate the presentation of the mathematical method and analysis, we combine the rudder and propeller forces as in \cite{InitialPaper} into a simpler `thruster force', which gives
\begin{equation}
	X = X_H + \XT\cos{\eta},\quad
	Y = Y_H + \XT\sin{\eta}, \label{thruster_forces} \quad
	N = N_H + \bar{x}_T\XT\sin{\eta}.
\end{equation}
The thruster acts on the hull at a longitudinal position $\bar{x}_T\in[-\frac 1 2\L,\frac 1 2\L]$, measured from the midship of the marine craft towards the front, and exerts a {\black propulsion} force in {\black the steering} direction $\eta$ of amplitude given by the propeller force $\XT$; {\black see Fig. \ref{Ship_xT}. {\black This is similar to the rudder force in \cite{STC2007}, where $\XT$ is constant while we will consider nonlinear $\XT$.} 
	The variable $\eta$ represents the rudder angle and is an external input, which will be assumed to obey a specific control law design.} 
\begin{figure}
	\centering
	\includegraphics[width= 0.37\linewidth]{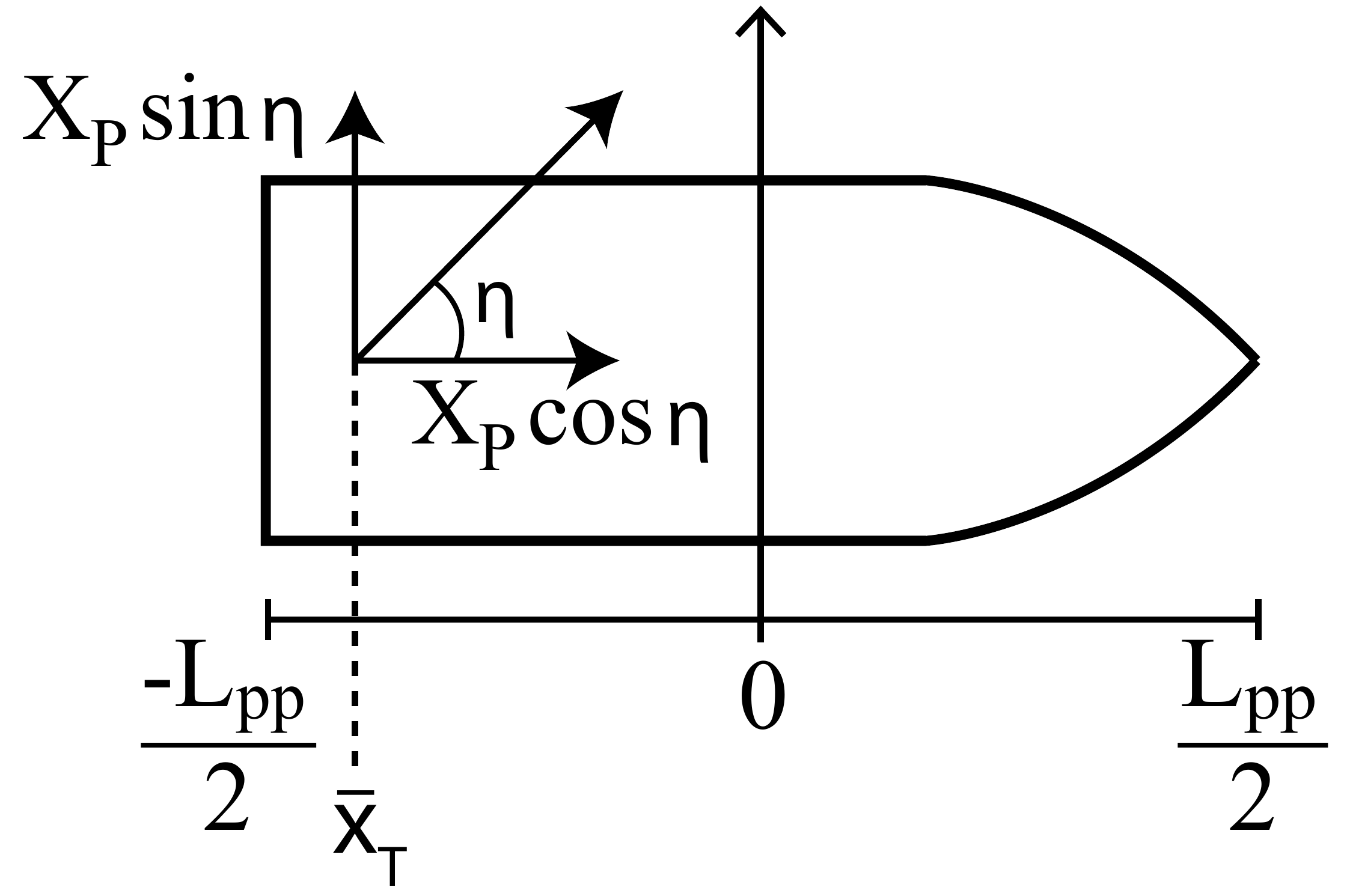}
	\caption{{\black Sketch of a ship model illustrating the location of the thruster and the thruster force components for the surge, $X_P\cos\eta$, and the sway, $X_P\sin\eta$.}}
	\label{Ship_xT}
\end{figure}

{\black For the hull forces we take the general non-dimensional form of \cite{StandShips2006,ToxopeusThesis}. Compared to \eqref{OriginalSys}, we scale $u,v$ by the ship length $\L$ to unit 1/s; the full non-dimensionalization follows below. With $V=\sqrt{u^2+v^2}$, the forces read}
\begin{align*}
	{X}_{H} =&  X_{u|u|} u^2 + X_{\beta \gamma} v{\w} ,\\
	{Y}_{H} =&  Y_{\beta} u v + Y_{\gamma} u{\w} + Y_{\beta |\beta|} v |v| + Y_{\gamma |\gamma|} {\w} |{\w}| 
	+ Y_{\beta |\gamma|} v |{\w}| + Y_{|\beta| \gamma} |v| {\w} +  Y_{ab} u^{3} |v|v V^{-3} ,\\
	{N}_{H} =&  N_{\beta} uv + N_{\gamma} u{\w} + N_{u' \gamma c} u|\w|\w V^{-1}   
	+ N_{\gamma |\gamma|} {\w} |{\w}| + N_{\beta|\beta|} v|v|   \\&
	+ N_{\beta \beta \gamma } {\w} v^2 V^{-1} + N_{\beta \gamma \gamma } v{\w}^2 V^{-1} + N_{ab} u v^{3} V^{-2}.
\end{align*}
{\black To ease the exposition, some exponents of terms in $Y_H$ and $N_H$ are already specified to the HTC from which we choose as default  all values of coefficients; see Appendix \ref{HTC_param}. Several nonlinear terms in} these forces implement the aforementioned second-order modulus form of the drag {\black analogous to \cite{StandShips2006,FossenHandbook,STC2007}. The corresponding non-dimensionalization, up to time, of \eqref{OriginalSys} has the same form, but with the dimensionless mass coefficients listed in Table~\ref{Mass_coeffs}. Within $N$ the thruster position turns into the dimensionless $x_T\in[-1/2, 1/2]$. As default value we take the rudder position of the HTC, which gives $x_T\approx -0.49$; see \cite{ToxopeusThesis}.}

The propeller force taken from \cite{ToxopeusThesis} {\black in the corresponding scaling reads} 
\begin{align}\label{e:XP}
	X_P = 2(1-t) \frac{\L}{T}{n_p}^2{D_p}^4 K_T, \quad K_T := \sum_{i=0}^{5}K_{Ti}\left(\frac{1-w}{n_pD_p}u\right)^i,
\end{align}
{\black where $n_p{>0}$ is the propeller frequency,  ${D}_p:=\bar{D}_p/\L$ is the non-dimensional propeller diameter, and all other parameters are by default those of the HTC; see Appendix \ref{HTC_param}. 
	The propeller frequency can be removed by a non-dimensionalization of time. Indeed, all forces scale quadratically with respect to the velocity $u=n_p\check{u}$, e.g., $X_P={n_p}^2\check{X}_P(\check{u})$ with $\check{X}_P$ independent of $n_p$. 
	This implies the natural relation that all velocities are proportional to the propeller speed: rescaling all velocities and time proportional to $n_p$ gives a factor ${n_p}^2$ on both sides of \eqref{OriginalSys} so that $n_p$ is removed upon division. Concerning the (dimensionless) propeller diameter $D_p$, we observe that scaling $u=D_p\tilde{u}$ gives $X_P={D_p}^4\widetilde{X}_P(\tilde{u})$, where $\widetilde{X}_P$ is independent of $D_p$. However, not all forces scale in the same way and, as we will explain below, the value of $D_p$ enters into the analysis of the model. 
	In the remainder of this paper, we discuss the fully non-dimensional equations using the same notation and fixing $n_p=1$ in \eqref{e:XP}.}

The model is completed by the control law for the steering angle $\eta${\black. Here we choose a standard combination $\eta =\eta_1+\eta_2$} of yaw damping {\black proportional control (P-control)} $\eta_1$ and yaw restoring P-control $\eta_2$; we refer to \cite{FossenHandbook} for a general background. The first consists of adding a proportional compensation to the yaw velocity $\w$ by setting $\eta_1 = \eps_\w(\w-\w_\target)$, with target yaw velocity $\w_\target$ and control parameter $\eps_\w\geq 0$. Analogously, for the yaw angle $\psi$ we have $ \eta_2 =  \eps_\psi(\psi-\psi_\target)$ with target yaw angle $\psi_\target$ and second control parameter $\eps_\psi\geq 0$. Notably, for $\eps_\psi>0$ this requires to add $\dot{\psi}=\w$ as a fourth equation to \eqref{OriginalSys}. 
{\black The control parameters $\eps_\w, \eps_\psi\geq 0$ are also referred to as control gains. For} controlling a straight line trajectory, {\black we have $r_\target=0$ and, since \eqref{OriginalSys} is otherwise independent of $\psi$, we may set $\psi_\target=0$. This gives}
\begin{equation}
	\eta = \eps_\w\w+\eps_\psi\psi,
	\label{Pcontrol}
\end{equation}
{\black which is a standard P-controller that also appears in, e.g., \cite{PAPOULIAS1994,STC2007,Tigkas2019}.}

In the following we study the non-dimensional $4$D `thruster model'
\begin{equation}
	\begin{pmatrix}
		m+m_{uu} 	& 0 		& 0 		&0\\
		0 			& m+m_{vv}	& m_{v\w} 	&0\\
		0 			& m_{\w v}	&I_z+ m_{\w\w}	&0\\
		0			& 0			& 0 		&1
	\end{pmatrix}
	\begin{pmatrix}
		\dot{u}\\
		\dot{v}\\
		\dot{\w}\\
		\dot{\psi}\\
	\end{pmatrix}=
	\begin{pmatrix}
		m v\w+X_H+\tau(u)\cos\eta\\
		-m {u}\w+Y_H+\tau(u)\sin\eta\\
		N_H+x_T\tau(u)\sin\eta\\
		\w
	\end{pmatrix},
	\label{4DSystem}
\end{equation}
{\black where we denote $\tau(u):=X_p$ for later convenience. For $\eps_\psi=0$ we also consider the 3D reduction to the invariant $(u,v,r)$-subsystem, i.e., the non-dimensionalized form of \eqref{OriginalSys}. It turns out to be relevant that the propeller force is monotone decreasing in the surge $u{>0}$, i.e.,
	\begin{equation}\label{e:xp_prime}
		\partial_u \tau(u) < 0.
	\end{equation}
	This is equivalent to $\partial_u K_T < 0$ for the HTC since in \eqref{e:XP} we have $t,w<1$. The fact that $\partial_u K_T < 0$ holds for $u>0$ is not obvious, but it can be verified numerically for the HTC. 
	
	We write \eqref{4DSystem} more compactly as $M\dot{\bv}=F(\bv)$, where $M$ is the matrix on the left hand side of \eqref{4DSystem},  containing the non-dimensionalized mass matrix from \eqref{OriginalSys}, $F(\bv)$ is the right hand side, and $\bv=(u,v,r,\psi)^\intercal$. Since $M$ is invertible,} the equivalent explicit compact form of \eqref{4DSystem} reads 
\begin{equation}\label{e:4DexplicitODE}
	\dot{\bv} = M^{-1} F(\bv).
\end{equation}

\section{Theoretical Analysis}\label{Theoretical_Analysis}

In this section, we analyze the impact of the yaw damping and yaw restoring control \eqref{Pcontrol} on the stability and bifurcation of {\black the equilibrium} straight motion. 
We include variations of the selected design parameters $D_p$, the propeller diameter, and $x_T$, the thruster position, in order to illustrate the methodology. We start discussing the existence of the straight motion as an equilibrium in the ship-fixed coordinates, then turn to the linear stability, and finally analyze the resulting bifurcations.

\subsection{Equilibrium Straight Motion}
\label{Stability_s_m}

The straight motion of the ship with constant speed corresponds to 
an equilibrium point given by $(u_0,v_0,\w_0,\psi_0)=(u_0,0,0,0)$, with $u_0>0$ and the reference direction $\psi_0=0$ of system \eqref{4DSystem}. 
Equilibria are those $\bv=(u,v,r,\psi)^\intercal$ for which $F(\bv)$ in \eqref{e:4DexplicitODE} vanishes. Setting $v=\w=\psi=0$, the last three components of $F(\bv)$ become zero, which in fact holds for any $\psi_0$ if $\eps_\psi=0$, so that in this case we obtain a line of equilibrium motion in any direction. For $\eps_\psi\neq 0$, this is constrained to the reference direction $\psi_0=0$. The remaining first component of $F(\bv)$ now reads $X_H + \tau(u)\cos(0)$, which gives the condition for the equilibrium straight velocity $u_0$ as
\begin{equation}
	\label{u_eq}
	X_{u|u|}{u_0}^2+\tau(u_0)=0,
\end{equation}
independent of the control parameters $\eps_\w, \eps_\psi$. This equation possesses a unique positive solution if \eqref{e:xp_prime} holds, i.e., $\partial_u\tau(u)<0$, since $\tau(0)>0$ and $X_{u|u|}<0$, and therefore the left hand side is strictly decreasing for $u>0$. These conditions hold for the HTC values {\black (see Appendix \ref{HTC_param})}, and then $u_0=u_0^\mathrm{HTC}\approx {\black 0.03}$ 
solves \eqref{u_eq} uniquely. {\black It also follows from \eqref{u_eq} that $\tau(u_0)>0$.}

Regarding $D_p$ and $x_T$, the equilibrium location is independent of $x_T$ since this does not appear in \eqref{u_eq}. For the propeller diameter $D_p$, we scale $u=D_p\tu$ and $\tilde{\tau}(\tu):= {D_p}^{-4}\tau({D_p\tu})$, which is independent of $D_p$. Then \eqref{u_eq} becomes $X_{u|u|}{\tu_0\hspace{0.0cm}}^2+{D_p}^2\tilde{\tau}(\tu_0)=0$, where the first {addend} is independent of $D_p$. Hence, the equilibrium depends on $D_p$, but 
for large values its location is approximately proportional to $D_p$. Indeed, the rescaled \eqref{u_eq},
upon multiplication by $({D_p}^2X_{u|u|})^{-1}$, takes the form
\begin{equation}
	{D_p}^{-2}{\tu_0\hspace{0.0cm}}^2+C_0+P(\tu_0)=0,
	\label{eq_u_Dp}
\end{equation}
with $P(\tu_0):=C_1\tu_0+\cdots +C_5{\tu_0\hspace{0.0cm}}^5$, where the $C_i$ are real constants. 
In the limit $D_p\to\infty$, the term ${D_p}^{-2}{\tu_0\hspace{0.0cm}}^2$ vanishes, and $\tu_0$ converges to the unique positive solution {\black $\tu_0^*$} of $C_0+P(\tu_0)=0$, which is independent of $D_p${\black . 
	Since $u=D_p\tu$, it follows that $u_0 \approx D_p {\tu_0}^*$ for $D_p\gg 1$. }
Condition \eqref{e:xp_prime}
is readily verified numerically for any $D_p$; for large $D_p$ we can also see this rigorously by inspecting the expression for $\partial_u \tau(u_0)$ as a function of $D_p$, which reads
$$ \partial_u \tau(u_0)(D_p)=\tC_5 {D_p}^3\left( -\tC_4\left(\frac{u_0}{D_p}\right)^4 + \tC_3\left(\frac{u_0}{D_p}\right)^3 - \tC_2\left(\frac{u_0}{D_p}\right)^2 + \tC_1\frac{u_0}{D_p} - \tC_0 \right).$$
Here all $\tC_i$ are positive constants. Thus, for $u_0$ exactly proportional to $D_p$ we have $\partial_u \tau(u_0)(D_p)={\black -}\tC {D_p}^3$, with $\tC{\black >}0$, and the approximate proportionality implies $\partial_u \tau(u_0)<0$ for $D_p\gg 1$.

\subsection{Stability of Straight Motion}\label{s:linstab}

We first recall from \cite{InitialPaper} that the straight motion without control is linearly unstable. Indeed, {\black $F'(\bv)$ from \eqref{e:4DexplicitODE} at $\bv=(u_{0},0,0,0)^\intercal$ 
	reads
	\begin{align*}J&:=
		\begin{pmatrix}
			2X_{u|u|}{u_0}+\partial_u \tau(u_0)		&0				&0	&0\\
			0									&Y_\beta {u_0}	&(Y_\gamma-\mL)u_0+\tau(u_0)\eps_\w	&\tau(u_0)\eps_\psi\\
			0									&N_\beta u_0		&N_\gamma{u_0}+x_T\tau(u_0)\eps_\w	&x_T\tau(u_0)\eps_\psi\\
			0									&0	&1	&0\\
		\end{pmatrix}.
	\end{align*}
	At $\eps_\psi=0$ it has vanishing fourth column, and hence, a zero eigenvalue, which is the same for $M^{-1}J$. The remaining eigenvalues of $M^{-1}J$ are those of its upper left $3\times 3$-submatrix, which we denote by ${S}$. This} has block structure with upper left entry 
 $\lambda_1=(2X_{u|u|}{u_0}+\partial_u \tau(u_0))/(m+m_{uu})$, {\black which is always negative since $X_{u|u|}u_0<0$ and $\partial_u\tau(u_0)<0$, as discussed above.} 
Thus, the linear stability of the equilibrium point $(u_0,0,0)$ is determined by the lower right $2\times 2$-submatrix of ${S}$. For the HTC, {\black and at $\eps_\w=0$,} we readily find that its eigenvalues are non-zero with opposite signs; cf.\ \cite{InitialPaper}. %
{\black Due to the second-order modulus terms,} the nonlinear terms in $F$ are differentiable at $v=r=0$, but not in a neighborhood of this. However, \eqref{e:4DexplicitODE} can be cast as a semi-linear {\black ordinary differential equation since $F(\bv)=J\bv+h(\bv)$, where $h(\bv)=\calO(\|\bv\|^2)$. Therefore,} the linear stability principle {\black applies}; see also \cite{SteinMacher2020}. Hence, the straight motion is unstable for the HTC {\black without control} and the free parameters $D_p, x_T$ do not change this instability: $x_T$ does not influence the stability analysis since the matrix and equilibrium do not depend on it, and the propeller diameter $D_p$ enters only in $\partial_u \tau(u_0)$, which modifies the eigenvalue $\lambda_1$ of ${S}$, but -- as noted above -- it is negative for any value of $D_p$.

In the remainder of this section, we analyze the eigenvalues of {\black $M^{-1}J$ when including} the P-control \eqref{Pcontrol}. 
With $p_{11} = \lambda_1$, it has the form
\begin{equation}A:=M^{-1}J= 
	\begin{pmatrix}
		p_{11}	&0		&0		&0		\\
		0		&p_{22}	&p_{23}	&p_{24}	\\
		0		&p_{32}	&p_{33}	&p_{34}	\\
		0		&0		&1		&0
	\end{pmatrix},
	\label{Matrix_A}
\end{equation}
and we {\black write} the other matrix entries as follows, noting the dependencies on $u_0$, $\eps_\psi$, $\eps_\w$:
\begin{center}
	\begin{tabular}{ c c c } 
		$p_{22}=p_{22u}u_0$, 	& $p_{23}=p_{23u}u_0+q_{23}\tau(u_0)\eps_\w$, & $p_{24}=q_{23}\tau(u_0)\eps_\psi$,\\ 
		$p_{32}=p_{32u}u_0$, 	& $p_{33}=p_{33u}u_0+q_{33}\tau(u_0)\eps_\w$, & $p_{34}=q_{33}\tau(u_0)\eps_\psi$,
	\end{tabular}
\end{center}
\medskip
with $D:=(m+m_{vv})(I_z+m_{rr})-m_{rv}m_{vr}$, and
\begin{align*}
	p_{22u} &= D^{-1}\big((I_z+m_{rr})Y_\beta -m_{vr}N_\beta\big), & 
	p_{32u} &= D^{-1}\big(-m_{rv}Y_\beta+(m+m_{vv})N_\beta\big),\\
	p_{23u} &=D^{-1}\big( (I_z+m_{rr})(Y_\gamma-\mL)-m_{vr}N_\gamma\big), & 
	q_{23} &= D^{-1}\big(I_z+m_{rr}-m_{vr}x_T\big),\\
	p_{33u} &= D^{-1}\big(-m_{rv}(Y_\gamma-\mL)+(m+m_{vv})N_\gamma\big), & 
	q_{33} &= D^{-1}\big(-m_{rv}+(m+m_{vv})x_T\big).
\end{align*}
Further we define the following {\black coefficients}, which enter in the stability result:
\begin{align}
	\begin{split}
		K_{11} &= {q_{33}}^2\tau(u_0)^2, \\
		K_{02} &= q_{33}(p_{32u}q_{23}-p_{22u}q_{33})u_0\tau(u_0)^2, \\
		K_{01} &= \left[(p_{22u}+p_{33u})(p_{32u}q_{23}-p_{22u}q_{33})+q_{33}(p_{23u}p_{32u}-p_{22u}p_{33u})\right]{u_0}^2\tau(u_0), \\ 
		K_{10} &= \left[(p_{22u}+p_{33u})q_{33}+p_{32u}q_{23}-p_{22u}q_{33}\right] u_0\tau(u_0), \\
		K_{00} &= (p_{22u}+p_{33u})(p_{23u}p_{32u}-p_{22u}p_{33u}){u_0}^3.
	\end{split}
	\label{Ks}
\end{align}

With these preparations we can formulate our main result concerning the change of stability of the unstable straight motion equilibrium for the HTC values. This is a refinement of the result in \cite{MathiasT}, and implies `global controllability' of the straight motion{\black . This means} that stabilization by the P-control is possible along any direction in the control parameter space, i.e., the positive quadrant of the $(\eps_\w, \eps_\psi)$-plane. {\black In \S\ref{s:npDp} we will find that this remains valid for any $D_p>0$, and in \S\ref{s:thrusterpos} we will study the non-trivial impact of $x_T$.} 

{\black By stabilization of the straight motion we mean exponential asymptotic stability, up to symmetry in $\psi$ for $\eps_\psi=0$, of the equilibrium $(u,v,r,\psi)=(u_0,0,0,0)$ in \eqref{4DSystem}. 
	Due to the linear stability principle this is equivalent to strictly negative real parts of the eigenvalues of $A$, or its 3D reduction $S$ for the case $\eps_\psi=0$.}   

\begin{theorem}
	\label{Thm_eps_curve}
	Consider the thruster model \eqref{4DSystem} with the HTC values and define
	\begin{equation}
		\eps_\psi(\eps_\w) := -\frac{K_{02}{\eps_\w}^2+K_{01}\eps_\w+K_{00}}{K_{11}\eps_\w+K_{10}},
		\label{function_eps}
	\end{equation}
	with $K_{ij}$ from \eqref{Ks} and $\eps_r, \eps_\psi$ the P-control parameters from \eqref{Pcontrol}. 
	Fix any $\epsrd, \epspsid\geq 0$ and for $s\geq 0$ consider control parameters on the ray 
	$s\cdot (\epsrd,\epspsid)$. 
	Then, as $s$ increases, the equilibrium $(u_0,0,0,0)$ of \eqref{4DSystem} is stabilized when $s\cdot (\epsrd,\epspsid)$ crosses the curve defined by \eqref{function_eps}.
	This crossing point lies at a unique $s^*>0$ {\black for each fixed $\epsrd,\epspsid$. T}he eigenvalues of {\black $A$} 
	{\black behave as follows:}  
	for {$\epspsid>0$}, 
	a complex pair of eigenvalues traverses the imaginary axis as $s$ crosses $s^*$, while for 
	{$\epspsid=0$} 
	one eigenvalue is fixed at zero and a simple real eigenvalue traverses zero at $s=s^*$. 
\end{theorem}
The theorem states that the curve defined by \eqref{function_eps} is a stability boundary for the control parameters, which is sometimes referred to as gain margin, e.g., \cite{PAPOULIAS1994}. 
We plot this stability boundary in Fig.~\ref{eps_curve_m}, and illustrate the organization of the eigenvalues of $A$ in the positive quadrant of the $(\eps_r,\eps_\psi)$-plane.
This eigenvalue configuration at $\eps_\psi=0$ and $\eps_{\w}$ on the stability boundary is reminiscent of a Bogdanov--Takens point. However, unlike a generic  unfolding, 
for $\eps_\psi=0$ the right-hand side of \eqref{e:4DexplicitODE} is independent of $\psi$ so that $(u,v,r,\psi)=(u_0,0,0,\psi)$ with arbitrary $\psi$ forms a line of equilibria. 
We discuss aspects of the resulting bifurcations in \S\ref{s:pitch} and \S\ref{Numerical_Bif_An}. The Hopf bifurcation analysis on the stability boundary for $\eps_\psi>0$ will be presented in \S\ref{s:hopf}. For the nonsmooth system, this analysis is more delicate than usual.

\begin{proof}
	We analyze the eigenvalues of the linear part of \eqref{4DSystem} given by $A$ in \eqref{Matrix_A}. As noted above, $p_{11}<0$ for the HTC values, so that it suffices to consider the lower right $3\times 3$-matrix, which we denote by $P$. 
	Its characteristic polynomial reads 
	\begin{equation}\label{e:charpoly}
		Q_P(\lambda)=\det(\lambda I-P)=\lambda^3+c_2\lambda^2+c_1\lambda+c_0,
	\end{equation}
	with $I$ the 3-by-3 identity matrix and 
	\begin{align*}
		c_0 &= p_{34}p_{22}-p_{24}p_{32} = (p_{22u}q_{33} - p_{32u}q_{23})u_0\tau(u_0)\eps_\psi,\\
		c_1 &= p_{22}p_{33}-p_{23}p_{32}-p_{34} \\
		&= (p_{22u}p_{33u}-p_{23u}p_{32u}){u_0}^2 + (p_{22u}q_{33}-p_{32u}q_{23})u_0\tau(u_0)\eps_\w - q_{33}\tau(u_0)\eps_\psi,\\
		c_2 &= -p_{22}-p_{33} = -(p_{22u}+p_{33u})u_0 - q_{33}\tau(u_0)\eps_\w.
	\end{align*}
	The Routh-Hurwitz criterion, see \cite{Gantmacher}, states that all eigenvalues of $P$ have negative real part if and only if $c_2, c_0>0$ and $c_2c_1-c_0>0$. The first condition is always satisfied for the HTC values and $\eps_\psi>0$ since $p_{22},p_{33}<0$, which implies $c_2>0${\black. In addition,} $p_{24}>0$, $p_{32},p_{34},p_{22}<0$ yield $c_0>0$. Concerning the second condition, using $K_{ij}$ from \eqref{Ks}, {\black a direct computation gives}
	\begin{align*}
		c_2c_1-c_0 &= 
		K_{11}\eps_\psi\eps_\w + K_{02}{\eps_\w}^2 + K_{01}\eps_\w + K_{10}\eps_\psi + K_{00}.
	\end{align*}
	{\black Since $\partial_{\eps_\psi} c_0>0$ for the HTC}, precisely those $(\eps_\w, \eps_\psi)$ `above' the convex curve defined by $c_2c_1-c_0=0$, or equivalently \eqref{function_eps}, provide eigenvalues with negative real part.
	In addition, for the control values satisfying \eqref{function_eps} {\black and} $\eps_\psi(\eps_\w)\neq 0$, it follows that there exist{\black s} a pair of complex conjugates with vanishing real part. Indeed, for $c_2c_1-c_0=0$ the characteristic polynomial can be factorized as $Q_P(\lambda)=(c_2\lambda^2+c_0)\left(\frac{1}{c_2}\lambda+1\right)$, and since $c_2, c_0>0$, the eigenvalues of the first factor correspond to a pair of purely complex conjugates, $\lambda_\pm=\pm\sqrt{-c_1}$.
	
	Finally, {\black at $\eps_\psi=0$ the last column of the matrix \eqref{Matrix_A} vanishes, which gives the fixed zero eigenvalue (and $c_0=0$)}. The conditions for the other eigenvalues to have {negative} real parts are $c_2,c_1>0$. 
	{\black Here $c_2>0$ holds as above since $p_{22},p_{33}<0$ for the HTC values and $\eps_{\w}\geq 0$.
		On the one hand, $c_1$ is linear in $\eps_\w$ with positive slope and $c_1<0$ at $\eps_{\w}=0$ for the HTC. Hence, $c_1$ changes sign from negative to positive at a unique $\eps_{\w_1}>0$, and at $\eps_{\w_1}$ the characteristic polynomial reads $Q_P(\lambda)=\lambda^2(\lambda+c_2)$, with a double root. On the other hand, since $c_0=0$ and $c_2>0$, the sign of $c_1$ is that of $c_2c_1-c_0$ so that $\eps_\psi(\eps_{\w})=0$ is the stability threshold as claimed.}
\end{proof}

\begin{figure}
	\centering
	\includegraphics[width= 0.4\linewidth]{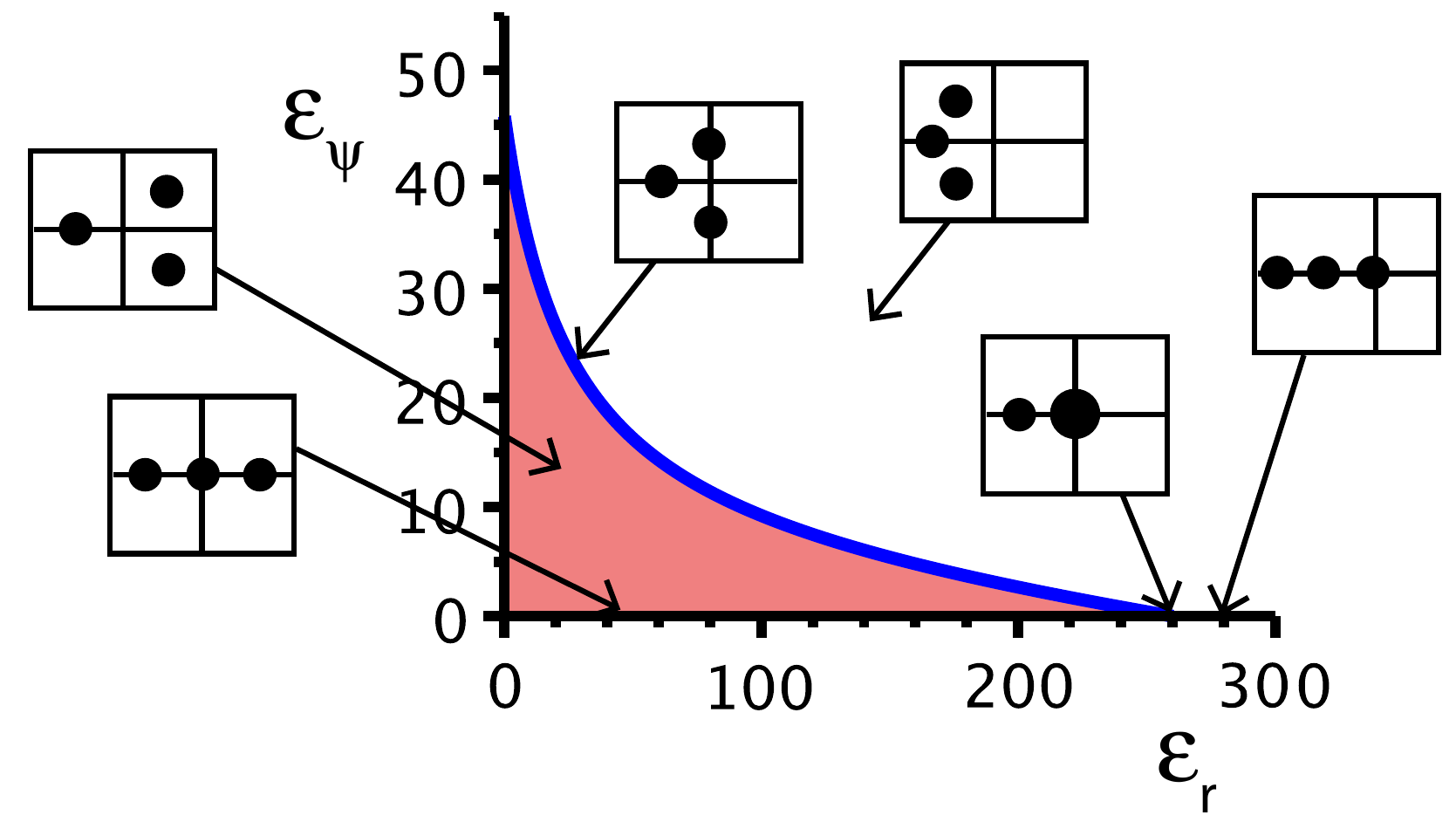}
	\caption{Stability boundary of the straight motion for the HTC values in the positive quadrant of the $(\eps_\w,\eps_\psi)$-plane. {\black The insets illustrate the arrangement of the eigenvalues 
			(in the complex-plane)		
			on the boundary, the $\eps_\w$-axis and nearby.} The straight motion is unstable in the colored area below the curve, and it is stable in the white region. The transition to stability for $\eps_\psi>0$ goes via a complex conjugate pair of eigenvalues, and the transition for $\eps_\psi=0$ by a real eigenvalue.}
	\label{eps_curve_m}
\end{figure}

Next we discuss the impact of varying the parameters $D_p$ and $x_T$ on the stabilization and thus linear controllability of the straight motion.

\subsubsection{Impact of changing the parameter $D_p$}\label{s:npDp}

We scale the propeller diameter $D_p$ as in \S\ref{Stability_s_m}: $u=D_p\tu$ and $\tau(u)={D_p}^4\tt(\tu)$, so $\tt$ is independent of $D_p$ 
{\black and \eqref{function_eps} becomes }
\begin{equation}
	\eps_\psi(\eps_\w,\tu;D_p) = -\frac{\tK_{02}{D_p}^6\tu\tt^2{\eps_\w}^2+\tK_{01}{D_p}^3\tu^2\tt\eps_\w+\tK_{00}{\tu}^3}{\tK_{11}{D_p}^5\tt^2\eps_\w+\tK_{10}{D_p}^2\tu\tt},
	\label{eps_psi_Dp}
\end{equation}
where $\tK_{ij}$ are real constants. 
{\black From \eqref{eq_u_Dp} we have $\tu_0=\tu_0(D_p) \to \tu_0^*\neq 0$, as $D_p\to \infty$.} 
Setting $E:=\tt {D_p}^3\eps_\w$, $\Upsilon:=\tt {D_p}^2$ {\black in \eqref{eps_psi_Dp}} gives 
an expression which has the same functional form with same signs of coefficients as \eqref{function_eps}. Thus, the stability boundary is qualitatively the same for different $D_p$. 
Specifically, on the one hand, for $\eps_\w=0$ we get from \eqref{eps_psi_Dp} that $\eps_\psi(0,\tu;D_p) = -\frac{\tK_{00}\tu^2}{\tK_{10}{D_p}^2\tt} = \calO({D_p}^{-2})$. On the other hand, for $\eps_\psi=0$ in \eqref{eps_psi_Dp} we obtain the function  
$\eps_\w(\tu;D_p)=\frac{\tu}{{D_p}^3\tt}\frac{-\tK_{01}\pm\sqrt{{\tK_{01}\hspace{0.0cm}}^2-4\tK_{02}\tK_{00}}}{2\tK_{02}}$,
and $\tu_0(D_p) \to \tu_0^*\neq 0$ implies $\eps_\w(\tu_0;D_p)= \calO({D_p}^{-3})$. 
{\black We conclude that increasing the propeller diameter $D_p$ enlarges the convex stable region in parameter space  with different rates along the different axes.} 
In fact, numerically {\black this} holds for all $D_p$.

\subsubsection{Impact of changing the thruster position parameter $x_T$}
\label{s:thrusterpos}

As shown next, understanding the impact of $x_T$ is more involved. {\black For instance, 
	taking} $x_T> 0.18$, the P-control cannot stabilize the straight motion at all, thus creating an `uncontrollable' situation. 

Before preparing the precise statement and proof, we first note that while \eqref{u_eq} does not depend on $x_T$, the function $\eps_\psi(\eps_\w)$ in \eqref{function_eps} does{\black, and we therefore denote this as $\eps_\psi(\eps_\w; x_T)$.} 
However, it is not clear that the graph of $\eps_\psi(\eps_\w;x_T)$ remains a stability boundary since this only accounts for one of the criteria for stable  eigenvalues. 
The Routh-Hurwitz criterion in the proof of Theorem \ref{Thm_eps_curve} {\black also} requires 
$c_2>0${\black,  but, for instance at $x_T=0.3$,} we have 
$c_2<0$ for $\eps_\w>\eps_{\w_2}\approx 144.47$; see Fig.~\ref{cases} (a). 
{\black I}n this case the curve $\eps_\psi(\eps_\w;x_T)$ is not a stability boundary, and the straight motion cannot be stabilized by the P-control \eqref{Pcontrol}.
\begin{figure}
	\centering
	\begin{tabular}{ccc}	
		\includegraphics[width=0.2\linewidth]{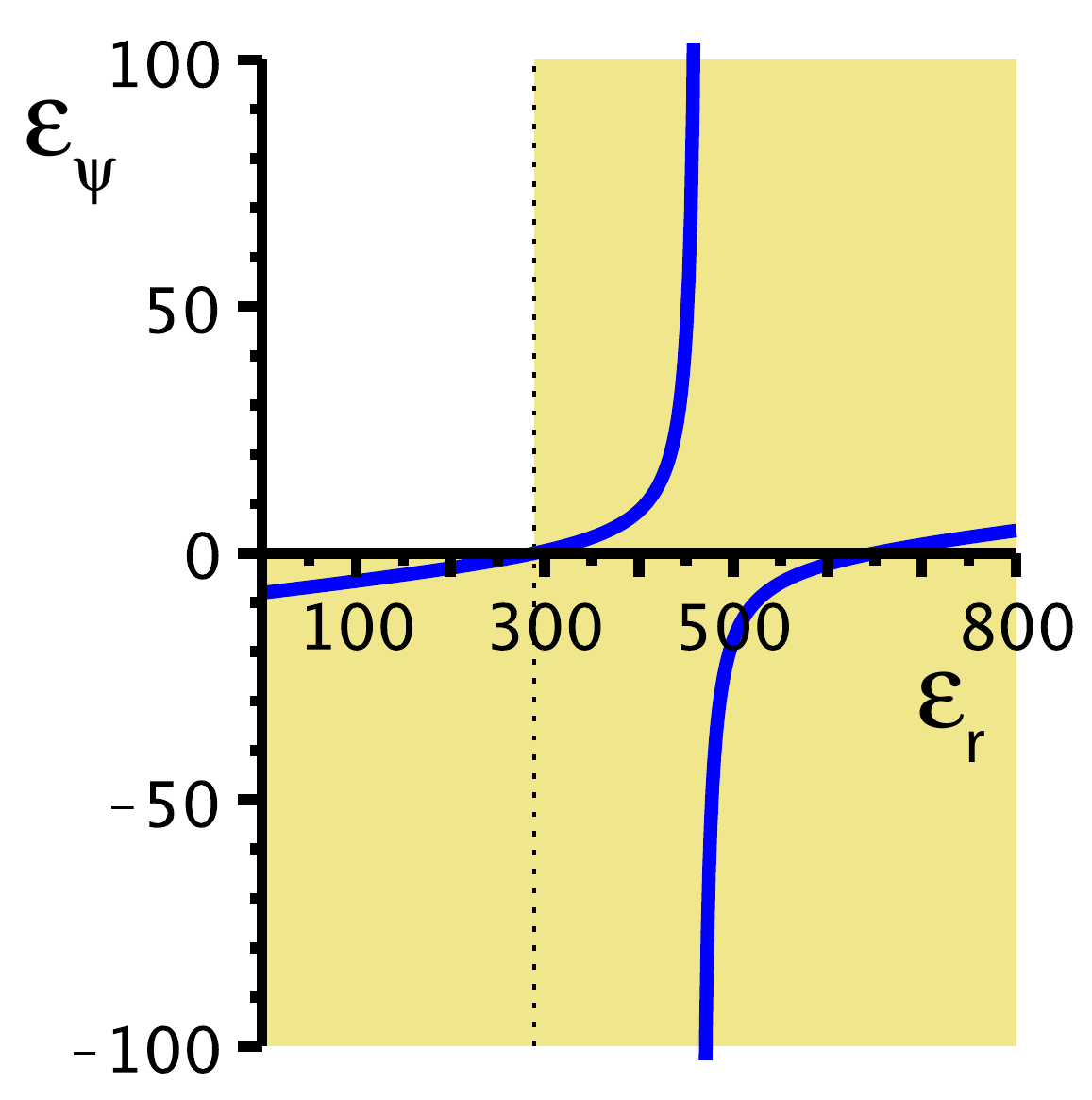} &
		\includegraphics[width=0.2\linewidth]{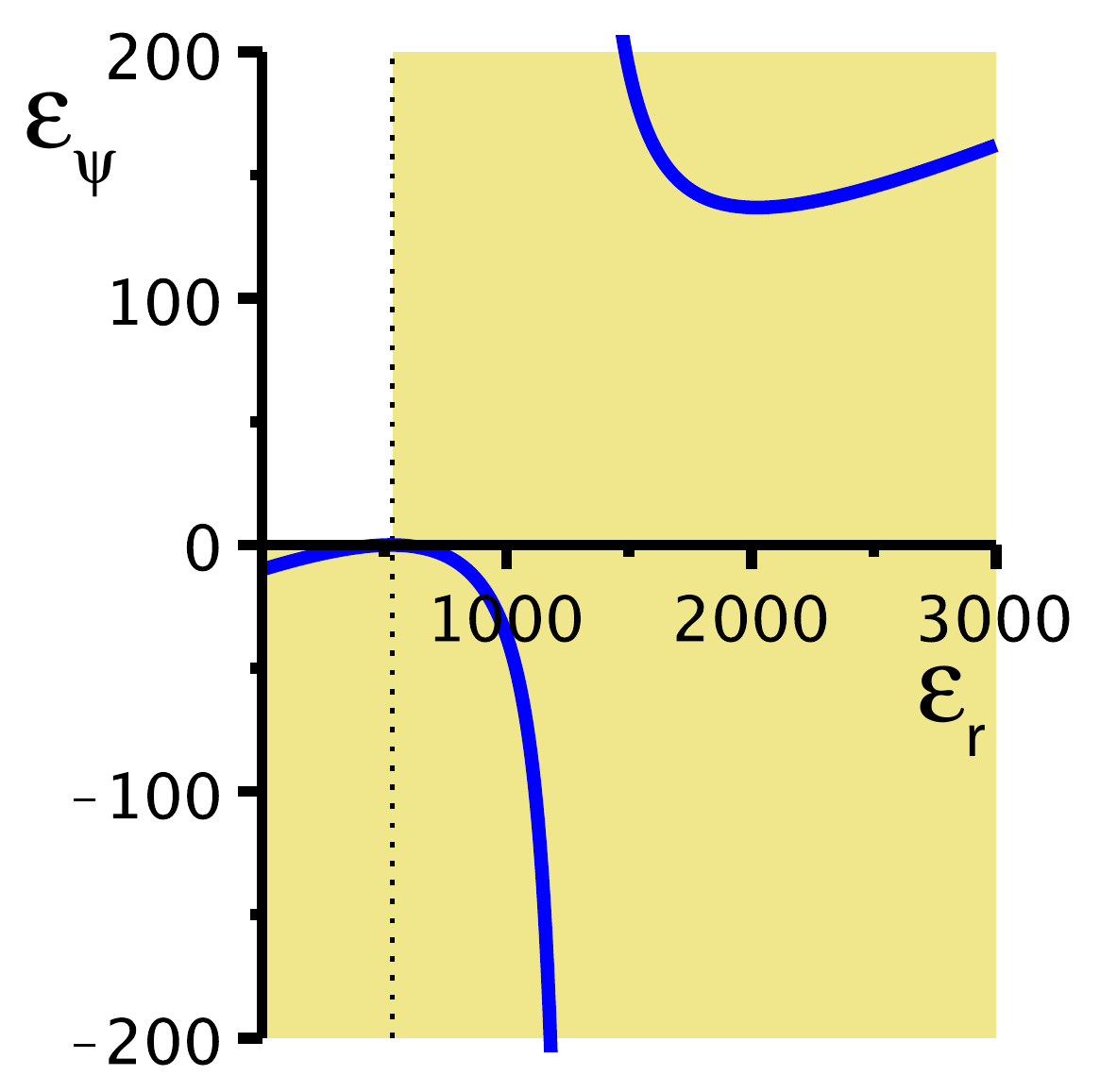} & 
		\includegraphics[width=0.2\linewidth]{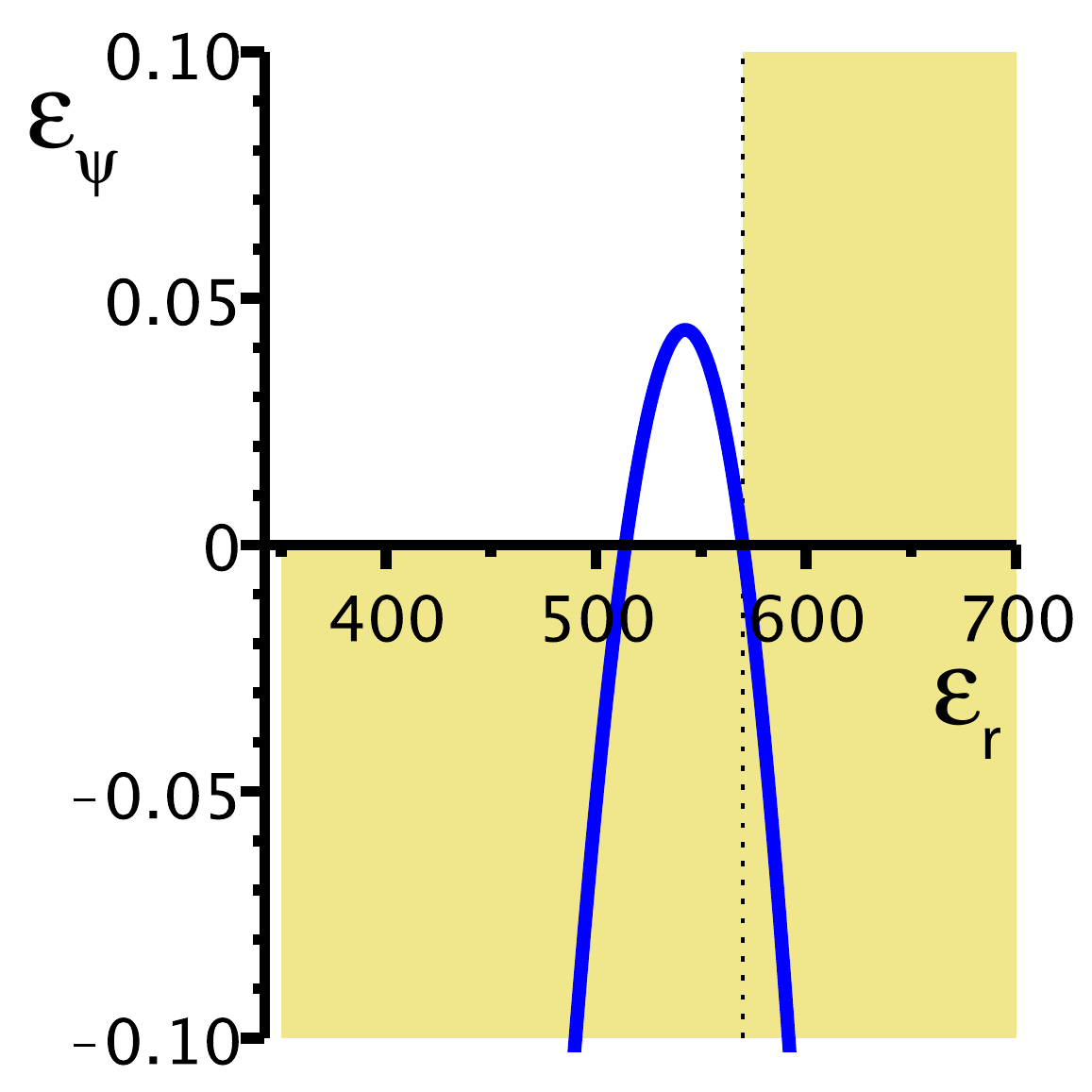} \\
		(a) & (b) & (c)
	\end{tabular}	
	\caption{We display different forms of the curve $\eps_\psi(\eps_\w{\black;x_T})$ (blue) together with violations of the stability criteria $c_0<0,c_2<0$ (colored regions), where $c_0<0$ for $\eps_\psi<0$, and $c_2<0$ for large $\eps_\w$. {\rm(a)} $x_T=0.3$, {\rm(b)} $x_T=\xTs\approx 0.17$, and {\rm(c)} $x_T=0.16$, where we omit the second 
		branch which is similar to that in {\rm(b)}.}
	\label{cases}
\end{figure}
In the following we provide a detailed (and somewhat tedious) analysis that explains all possibilities to stabilize the straight motion {\black for different $x_T$ under sign conditions that hold for the HTC}. 

The coefficients $c_0, c_1, c_2$ of the eigenvalue problem {\black \eqref{e:charpoly} can be written as}
\begin{equation}\label{e:ccoeffs}
	\begin{aligned}
		c_0 &= (\ctA-\ctB x_T)\eps_\psi, \\
		c_1 &= -\cZ + (\ctA-\ctB x_T)\eps_\w + (\cA-\cB x_T)\eps_\psi,\\
		c_2 &= \cC + (\cA-\cB x_T)\eps_\w, 
	\end{aligned}
\end{equation}
where
\begin{equation}
	\begin{aligned}
		&\begin{aligned}
			\cA &:= D^{-1}m_{rv}\tau(u_0) > 0,
			&\cB &:= D^{-1}(m+m_{vv})\tau(u_0) > 0,\\
			\ctA &:= {\black - D^{-1}N_\beta u_0\tau(u_0) }> 0,
			&\ctB &:= {\black D^{-1}Y_\beta u_0\tau(u_0) } > 0,\\
		\end{aligned}\\
		&\cC := D^{-1}\left[m_{vr}N_\beta-(I_z+m_{rr})Y_\beta  -(m+m_{vv})N_\gamma+m_{rv}(Y_\gamma-\mL)\right]u_0 > 0,\\
		&\cZ :={\black D^{-1}(m N_\beta + N_\gamma Y_\beta - N_\beta Y_\gamma){u_0}^2} > 0.
	\end{aligned}\label{e:signcond}
\end{equation}
{\black With these definitions, the function 
	\eqref{function_eps} reads}
{\black 
	\begin{equation}\label{e:epspsixT}
		\eps_\psi(\eps_r;x_T) = \frac{ \Big(\cC(\ctA-\ctB x_T)+(\cA-\cB x_T)(\cZ+(\ctA-\ctB x_T)\eps_\w)\Big)\eps_\w -\cC\cZ}
		{\ctA-\ctB x_T + (\cA-\cB x_T)(\cC+(\cA-\cB x_T)\eps_\w)}.
\end{equation}}
The denominator has a unique root with respect to $\eps_\w$ given by 
\begin{equation}
	\eps_\w^* =
	\frac{\ctA-\cC\cA-(\ctB-\cC\cB)x_T}{(\cA-\cB x_T)^2},
	\label{asy_xT}
\end{equation}
{\black i.e., $\eps_\psi(\cdot;x_T)$ possesses a singularity at $\eps_r=\eps_r^*$;} see Fig.~\ref{eps_curve_m_xr_n03} (a). For the HTC, large negative values of $x_T$ {\black yield} negative $\eps_\w^*<0$, a sign change occurs at $x_T\approx -0.38$, and {\black there is a} singularity near $x_T=0.016${\black ; see Fig.~\ref{roots_epsR}}. In particular, $\eps_\w^*<0$ for the HTC default value $x_T\approx -0.49$.
\begin{figure}
	\centering
	\includegraphics[width=0.4\linewidth]{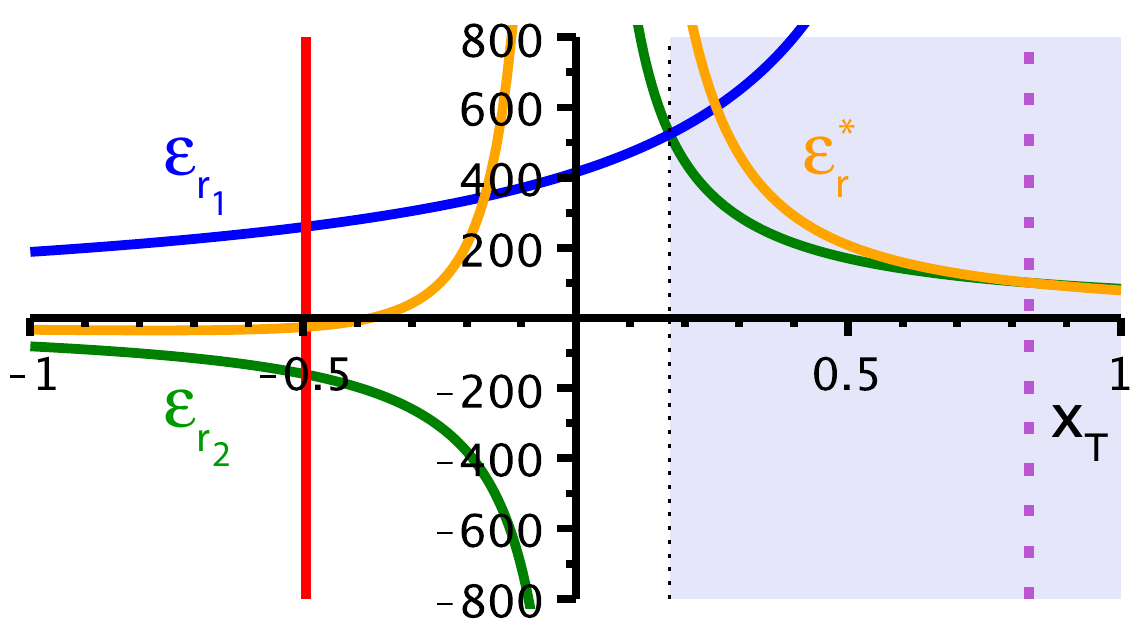}
	\caption{We plot relevant quantities {\black for the HTC} as functions of $x_T$. {\black We mark the default HTC value $x_T=-0.49429$ (red vertical line), the region $x_T>\xTs\approx 0.17$ (shaded area) and $x_T=x_{T_{1}}\approx 0.83$ (purple pointed line).} Orange curves: $\eps_\w^*(x_T)$ from \eqref{asy_xT} for the singularity of $\eps_\psi(\eps_\w;x_T)$; blue curve: $\eps_{\w_1}(x_T)$, the root of $c_1$ at $\eps_\psi=0$; green curves: $\eps_{\w_2}(x_T)$, the root of $c_2$. {\black The domain $[-0.5,0.5]$ of $x_T$ is extended to illustrate better the shape of the curves and the location of $x_{T_{1}}$.}}
	\label{roots_epsR}
\end{figure}
As long as $\eps_\w^*<0$, the singularity lies outside the positive range of the control parameters and is thus disregarded. However, for values of $x_T$ where $\eps_\w^*>0$, stabilization by $\eps_\psi$ alone is not possible. 
We plot an example in Fig.~\ref{eps_curve_m_xr_n03} (a). 
For $\eps_\w^*>0$, there {\black can} be two positive values of $\eps_\w$ for which $\eps_\psi(\eps_\w;x_T)=0$, i.e., potentially a stabilization and subsequent destabilization when increasing $\eps_\w$ from zero. We plot examples in Fig.~\ref{cases} (a) {\black and (c)}. 
In addition, comparing Figs.~\ref{eps_curve_m_xr_n03} (a) and \ref{cases} (a), a switching from convexity to concavity of $\eps_\psi(\eps_\w;x_T)$ has occurred. {\black For the HTC this lies} at $x_T\approx-0.17$; 
cf.\ Fig.~\ref{eps_curve_m_xr_n03} (b,c). 
{\black At this switching point \eqref{e:epspsixT} degenerates. This will be treated as a special case further below.} 

\begin{figure}
	\centering
	\begin{tabular}{cccc}
		\includegraphics[width=0.25\linewidth]{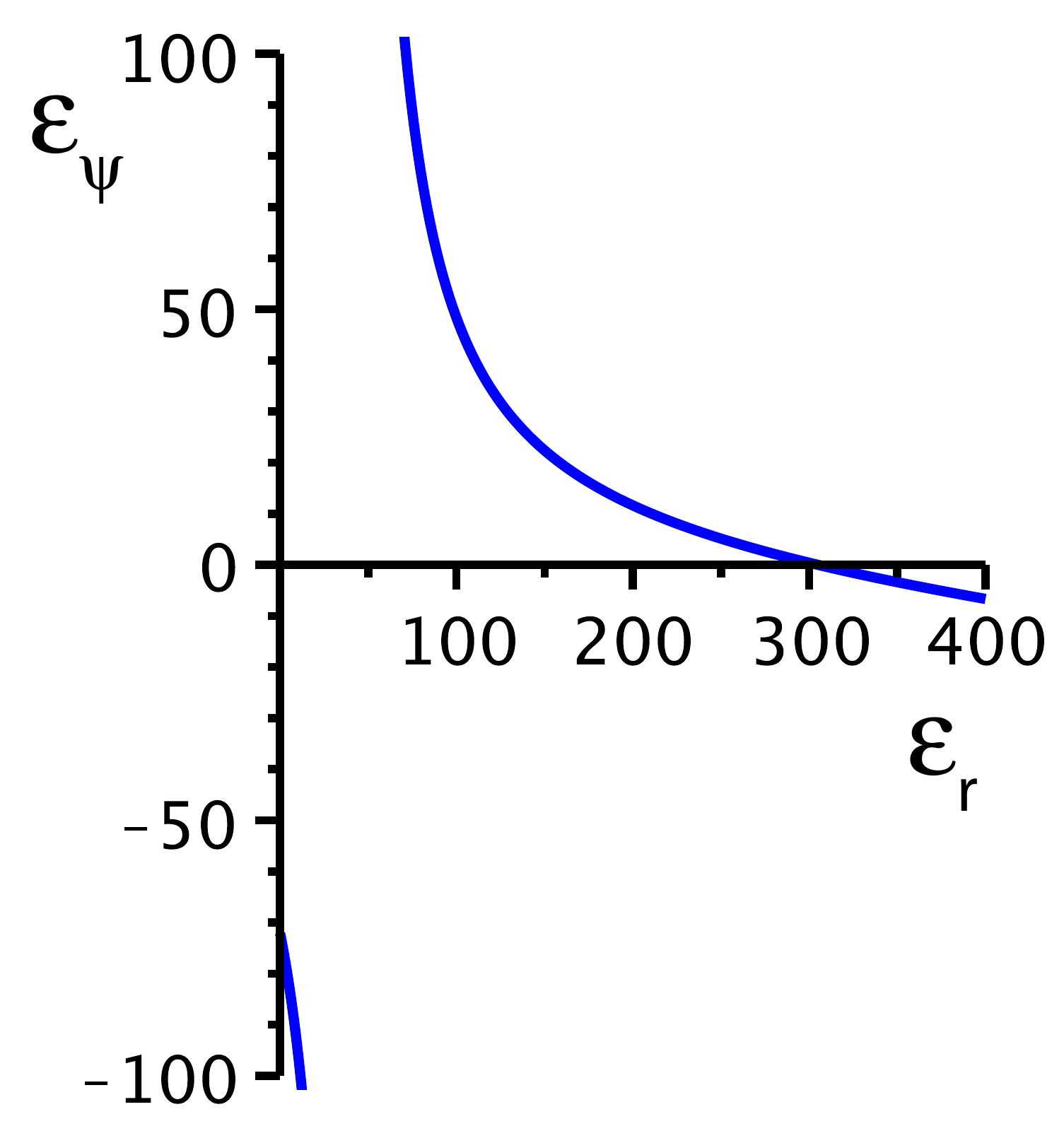} & 
		\includegraphics[width=0.23\linewidth]{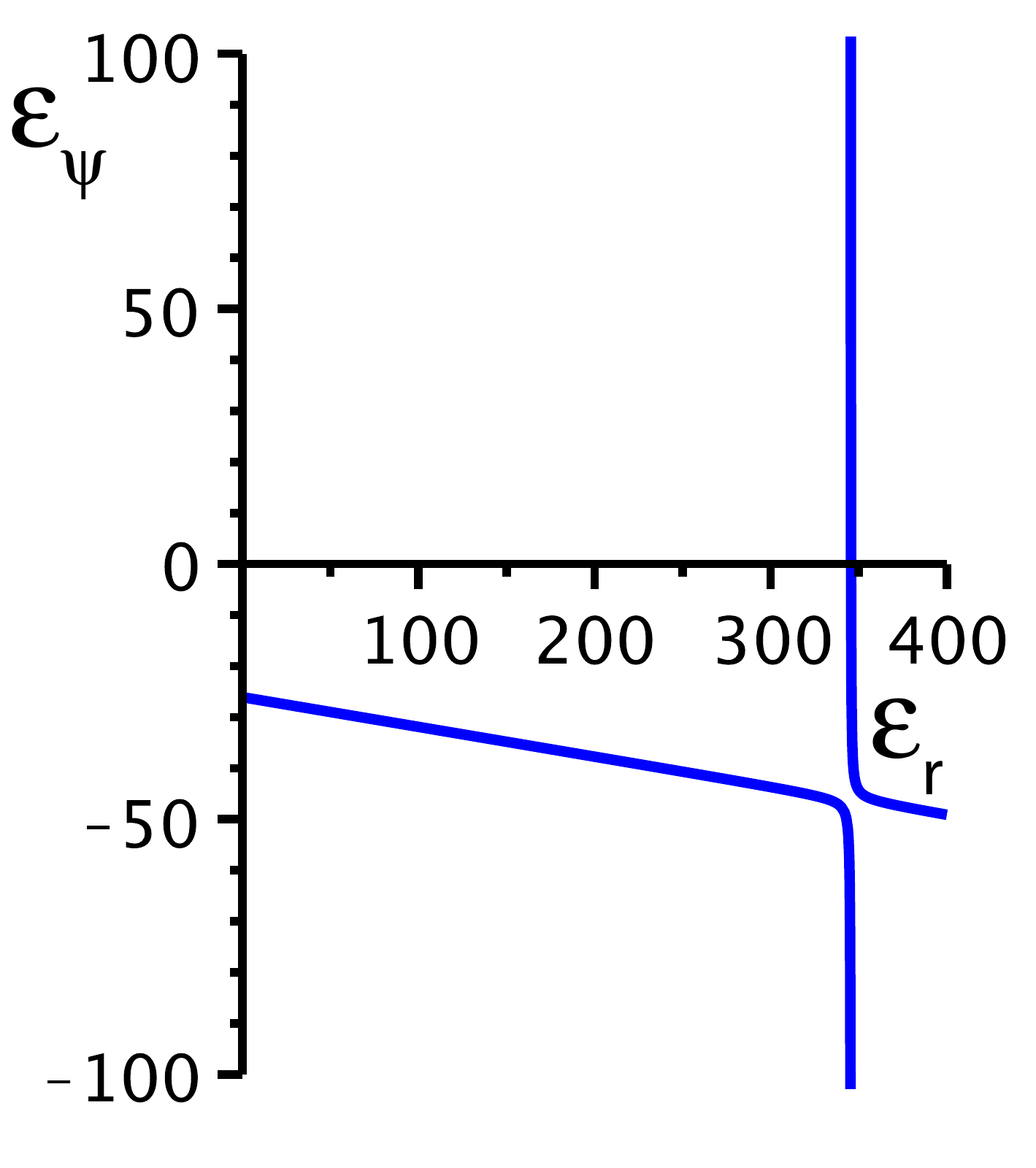} &
		\includegraphics[width=0.25\linewidth]{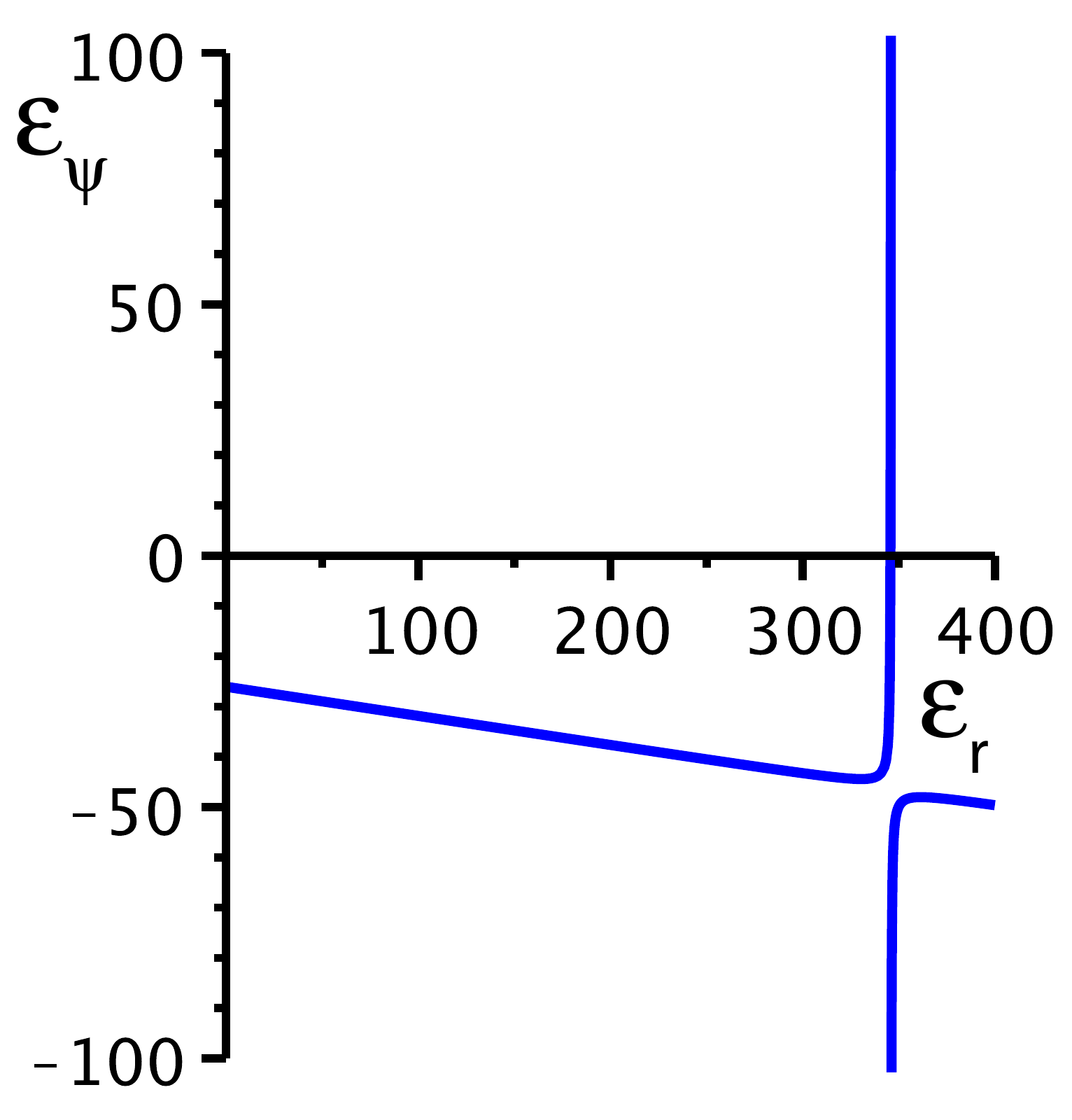}\\
		(a) & (b) & (c)
	\end{tabular}
	\caption{The function $\eps_\psi(\eps_\w;x_T)$ is plotted {\black for the HTC and} {\rm(a)} $x_T=-0.3$, where $\eps_\w^*\approx {\black 41.88}$;  
		{\rm(b)} $x_T=-0.1669$; and {\rm(c)} $x_T=-0.1668$.}
	\label{eps_curve_m_xr_n03}
\end{figure}

{\black Otherwise, whenever stabilization of the straight motion is possible, the stability region is bounded by the graph of $\eps_\psi(\eps_\w;x_T)$ from \eqref{e:epspsixT}, which gives the solutions to $c_2c_1=c_0$. 
	Its intersections with the $\eps_\w$-axis, if defined, are
	\[
	\eps_{\w_1}:=\frac{\cZ}{\ctA - \ctB x_T},\quad \eps_{\w_2}:=-\frac{\cC}{\cA-\cB x_T}.
	\]
	At these values and $\eps_\psi=0$ we have $c_1=0$ and $c_2=0$, respectively. For the HTC, $\eps_{\w_1}$ already appeared in the proof of Theorem~\ref{Thm_eps_curve}. If both are defined, so is $\eps_\w^*$, and we can write 
	\begin{equation}\label{e:epspsiroots}
		\eps_\psi(\eps_\w;x_T) = a\frac{(\eps_\w-\eps_{\w_1})(\eps_\w-\eps_{\w_2})}{\eps_\w-\eps_\w^*}, \qquad 
		a:=-\frac{\ctA-\ctB x_T}{\cA-\cB x_T}.
	\end{equation}
	For a given ordering of $\eps_{\w_1}, \eps_{\w_2}, \eps_{\w}^*$ and sign of $a$, we can read off the shape of the stability boundary as a function of $\eps_\w$. 
	The ordering depends on $x_T$, and relevant thresholds will be
	\begin{equation}
		x_{T_1}:=\frac{\ctA}{\ctB}=\frac{N_\beta}{Y_\beta}, \quad x_{T_2}:=\frac{\cA}{\cB} = \frac{m_{rv}}{m+m_{vv}}, \quad \xTs:=\frac{\cZ\cA+\cC\ctA}{\cZ\cB+\cC\ctB}, \\
		\label{e:xTthresholds}
	\end{equation}
	corresponding to the value of $x_T$ where $\eps_{\w_1}$ is undefined (and $c_0=0$), $\eps_{\w_2}$ is undefined, or $\eps_{\w_1}=\eps_{\w_2}$, respectively.
	Here $\xTs>0$ under the sign conditions \eqref{e:signcond}. 
	These special values of $x_T$ are independent of $u_0$ and thus of $D_p$. For the HTC the quantities in \eqref{e:xTthresholds} have the values noted in Fig.~\ref{sketch_xT2} and $x_{T_{1}}\approx 0.83$. In this figure we also plot the value{\black s $x_{T_0}$, $x_{T_-}$. The first solves $\eps_{\w}^*=0$ and has the explicit expression $x_{T_0}=(\cC\cA-\ctA)/(\cC\cB-\ctB)$. The threshold $x_{T_-}$} is the smallest solution to $\eps_{\w}^*=\eps_{\w_1}$ for the HTC. 
	This equation is generally quadratic with respect to $x_T$, but it is not clear that the roots are real under the conditions \eqref{e:signcond}. If these are real, then the smaller solution means a transition from convex to concave stability boundary for $\eps_\psi,\eps_\w> 0$.  
	
	Regarding signs and monotonicity we note that \eqref{e:signcond} implies, if $x_T$ is none of \eqref{e:xTthresholds}, 
	\begin{equation}\label{e:sgneps12}
		\begin{aligned}
			\sgn(\eps_{\w_1}) &=\sgn(x_{T_1}-x_T), && \sgn(a) = \sgn((x_{T_1}-x_T)(x_T-x_{T_2})), \\
			\sgn(\eps_{\w_2}) &=\sgn(x_T-x_{T_2}), && 
			\partial_{x_T}\eps_{\w_1}>0, \quad \partial_{x_T}\eps_{\w_2}<0,
		\end{aligned}
	\end{equation} 
	cf.\ Fig.~\ref{roots_epsR}.
}{\black In the case $\eps_{\w_1}=\eps_\w^*$ and $x_T\neq x_{T_2}$, the singularity of \eqref{e:epspsiroots} degenerates into a vertical graph that may form the stability boundary. Indeed, $c_1c_2=c_0$ can then be written as 
	\begin{equation}\label{e:degen}
		(\eps_\w-\eps_{\w_1})(\eps_\psi-a(\eps_\w-\eps_{\w_2}))=0, \quad \eps_{\w_1}=\eps_\w^*.
	\end{equation}
}
{\black With these preparations we present our main linear stability result.} As before, the stability region and its boundary refer to the positive quadrant in the $(\eps_\w,\eps_\psi$)-plane. 
\begin{theorem}
	\label{Thm_xT}
	Consider $x_T$ as a free parameter and assume for all other parameters the signs of coefficients within $c_0, c_1, c_2$ as {\black in \eqref{e:signcond} as well as $X_{u|u|}<0$, $\tau(0)>0$.} 
	Then two scenarios occur. In case {\black $x_{T_{2}}<x_{T_{1}}$}, 
	for any $x_T\leq {\black x_{T_{2}}}$ the statement of Theorem \ref{Thm_eps_curve} holds true for all 
	$\epsrd>0$. Furthermore, the stability boundary \eqref{function_eps} is a 
	{\black vertical line or a} 
	strictly monotone 
	{\black function of $\eps_\w$,} intersecting the $\eps_\w$-axis, but possibly not the $\eps_\psi$-axis with a vertical asymptote at ${\black \eps_\w=}\eps_\w^*$. 
	Moreover, if $x_T\geq\xTs$, the straight motion cannot be stabilized by any $\eps_\w, \eps_\psi>0$, while for $x_{T_{2}}<x_T<\xTs$, the stability region is bounded with boundary of parabolic shape, intersecting twice the $\eps_\w$-axis. 
	In case  {\black $x_{T_{1}}{\leq}x_{T_{2}}$}, the stability boundary is strictly monotone {\black or vertical} if $x_T<x_{T_{1}}$ and otherwise stabilization is impossible. \\
	{\black In all 
		cases which admit stabilization, the real part of the critical eigenvalue(s) has non-zero derivative as $(\eps_r,\eps_\psi)\geq 0$ transversally crosses the stability boundary along a curve.}
\end{theorem}

{\black In the final statement, transversal crossing means that the curve's tangent vector at the crossing point is linearly independent of the stability boundary's tangent vector.}

Before presenting the proof we proceed with some remarks. The theorem in particular implies the following alternative: either the straight motion can be stabilized by increasing $\eps_\w$ for all $\eps_\psi\geq 0$ or none, except in the case of a bounded stability region{\black; cf.\ Fig~\ref{sketch_xT2}}. In all cases, the possibility to stabilize is determined by the case $\eps_\psi=0$ alone. {\black The theorem implies that, given the signs in \eqref{e:signcond}, any sufficiently large $x_T$ makes it impossible to stabilize the straight motion by the P-control \eqref{Pcontrol}. But a necessary condition for this lack of controllability to be physically meaningful is $\xTs<1/2$ or $x_{T_{1}}<1/2$.} 
{\black Specifically, for the HTC values, we have $x_{T_{2}} < \xTs < x_{T_{1}}$ with $\xTs\approx 0.17$; cf.\ Fig.~\ref{sketch_xT2}.} 
{\black The decisive thresholds $x_{T_{1}}$ and $x_{T_{2}}$ have the surprisingly simple expressions $m_{rv}/(m+m_{vv})$ and $N_\beta/Y_\beta$, respectively, which depend purely on non-dimensional mass and hyodrodynamic coefficients of the hull.} 

\begin{figure}
	\centering
	\includegraphics[width=1.0\linewidth]{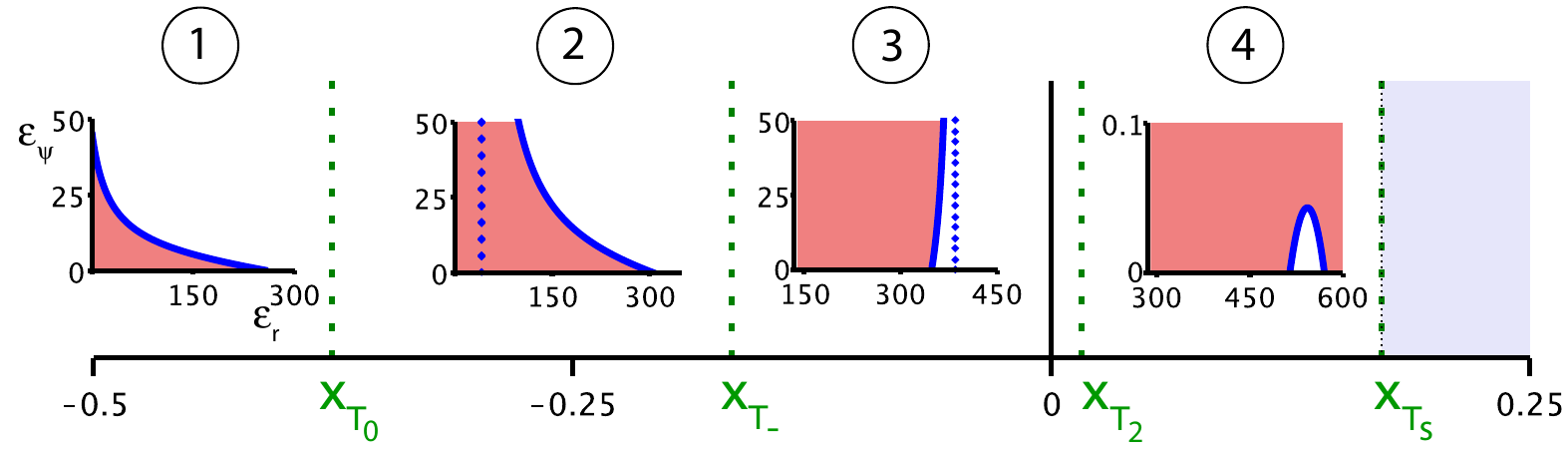}
	\caption{HTC parameters with varying $x_T$. Each inset shows the shape of the stability region in the $(\eps_\w,\eps_\psi)$-plane for $x_T$ in the corresponding interval on the $x_T$-axis. Unstable region is shaded red, blue line the stability boundary, vertical dashed line its vertical asymptote, if present. For $x_T>\xTs$ stabilization is not possible. Case 1: $x_T\in[-0.5,x_{T_0})$, $x_{T_0}\approx -0.38$, inset for $x_T\approx -0.49$. Case 2: $x_T\in(x_{T_0},x_{T_-})$, $x_{T_-}\approx -0.17$, inset for $x_T=-0.3$. Case 3: $x_T\in(x_{T_-},x_{T_2})$, $x_{T_{2}}\approx 0.016$, inset for $x_T\approx-0.16$. Case 4: $x_T\in(x_{T_2},\xTs)$, $x_{T_s}\approx 0.17$, inset for $x_T\approx 0.16$.}
	\label{sketch_xT2}
\end{figure}

\begin{proof}[Proof of Theorem~\ref{Thm_xT}]
	{\black If $X_{u|u|}<0$ and $\tau(0)>0$, then $u_0, \tau(u_0)>0$ by \eqref{u_eq}, which has a unique positive solution. The} Routh-Hurwitz criterion from the proof of Theorem \ref{Thm_eps_curve} consists of $c_0,c_2>0$ and $c_2c_1>c_0${\black; cf.\ \eqref{e:ccoeffs}. These conditions} must be attainable {\black with $\eps_\w\geq 0, \eps_\psi>0$ in order to} stabilize the straight motion, while for $\eps_\psi=0$ we have $c_0=0$ and the criterion becomes $c_1, c_2 >0$. The common condition $c_2>0$ is independent of $\eps_\psi$ and equivalent to either $x_T\leq x_{T_{2}}$ (and any $\eps_\w\geq 0$), or $\eps_\w< \eps_{\w_2}${\black; see Figure~\ref{cases}.}
	
	Concerning $\eps_\psi=0$, {\black since} $\ctA,\cZ,\ctB>0$, {\black we have} $c_1>0$ if and only if $x_T<x_{T_{1}}$ and $\eps_\w>\eps_{\w_1}$, {\black but only $\eps_{\w_1}\geq 0$ is a relevant constraint}. 
	In case $x_T>x_{T_{2}}$, {\black both $c_2>0$ and $c_1>0$ can be satisfied} if and only if $\eps_{\w_1}<\eps_{\w_2}$,  
	which is equivalent to $x_T< \xTs$. 
	Hence, controllability for $\eps_\psi=0$ requires $x_T<\min\left\{x_{T_{1}},\max\{x_{T_{2}},\xTs\}\right\}$,
	which covers the claimed statements in this case. 
	
	{\black We now turn to controllability for $\eps_\psi>0$. The condition $c_0>0$ is then equivalent to 
		$x_T<x_{T_{1}}$ as claimed for a possible stabilization. We may therefore assume $x_T<x_{T_{1}}$ in the following reasoning, which implies $\eps_{\w_1}>0$ by \eqref{e:sgneps12}.
		It remains to incorporate the last condition, $c_1c_2>c_0$, whose boundary is defined by $\eps_\psi(\eps_\w;x_T)$ from \eqref{e:epspsixT}. 
		{Since $\partial_{\eps_\psi}(c_1c_2-c_0)=0$ precisely at singularities of its graph, the stable region lies above or below this boundary }{\black unless it is vertical; cf. \eqref{e:degen}. We treat this case separately at the end and thus assume for now that $\eps_{\w_1}\neq \eps_\w^*$.}
		
		In case $x_T=x_{T_{2}}<x_{T_1}$, i.e., $\cA=\cB x_T$, we have 
		$\eps_\psi(\eps_\w;x_T) = \cC (\eps_\w -\cZ/(\ctA-\ctB x_T))$ with positive slope and root at $\eps_{\w_1}$ with respect to $\eps_\w$. This corresponds to the stability boundary for $\eps_\psi>0$ since $c_2, c_0>0$ due to the assumption $x_T=x_{T_2}<x_{T_1}$ and $\cC>0$. In this case $\partial_{\eps_\psi}(c_1c_2-c_0)=-(\cA-\cB x_T)<0$. Therefore, the stable region is below this boundary.
		
		The case {$x_T\neq x_{T_2}$} 
		is more involved and we use \eqref{e:epspsiroots}. Since $x_T<x_{T_1}$, both $\eps_{\w_1}$ and $\eps_{\w_2}$ exist, and 
		either $\eps_{\w_1}<\eps_\w^* < \eps_{\w_2}$ or $\eps_{\w}^*> \max\{\eps_{\w_1}, \eps_{\w_2}\}$ due to 
		\begin{equation}\label{e:rel2}
			\eps_\w^*=\eps_{\w_2}+(\ctA-\ctB x_T)/(\cA-\cB x_T)^2 > \eps_{\w_2}.
		\end{equation}
		A direct computation gives $\partial_{\eps_\w}\eps_\psi(\eps_{\w}; x_T)=0$ at $\eps_\w=\eps_\w^*\pm\sqrt{(\eps_\w^*-\eps_{\w_1})(\eps_\w^*-\eps_{\w_2})}$ so that the graph of $\eps_\psi(\cdot;x_T)$ has at most one critical point on either side of the singularity $\eps_\w=\eps_\w^*$. 
		Thus, in case $\eps_{\w_1}<\eps_\w^* < \eps_{\w_2}$ the branches of the graph are monotone on $(\eps_{\w_1},\eps_\w^*)$ and $(\eps_\w^*, \eps_{\w_2})$, respectively. In case $\eps_{\w}^*> \max\{\eps_{\w_1}, \eps_{\w_2}\}$ the branches are monotone on $(\max\{\eps_{\w_1}, \eps_{\w_2}\}, \eps_\w^*)$ and $(\eps_{\w}^*, \infty)$, respectively. 
		The type of monotonicity and the sign of $\eps_\psi(\cdot;x_T)$ can be inferred from the slopes at $\eps_{\w_1}, \eps_{\w_2}$, when given their ordering with respect to $\eps_\w^*$. 
		On the one hand, by a direct computation we have
		\begin{equation}
			\partial_{\eps_\w}\eps_\psi(\eps_{\w_1}; x_T) = a\frac{(\eps_{\w_1}-\eps_{\w_2})(\eps_{\w_1}-\eps_{\w}^*)}{(\eps_{\w_1}-\eps_{\w}^*)^2},
			\label{e:partialepspsi}
		\end{equation}
		which has the sign of $(x_T-x_{T_2})(\eps_{\w_1}-\eps_{\w_2})(\eps_{\w_1}-\eps_{\w}^*)$ 
		due to \eqref{e:sgneps12} and the assumption $x_T<x_{T_1}$. On the other hand, 
		we compute $\partial_{\eps_\w}\eps_\psi(\eps_{\w_2}; x_T)=(\cA-\cB x_T)(\eps_{\w_2}-\eps_{\w_1})$. 
		Together with the previous we conclude
		\begin{equation}\label{e:sgnpartials}
			\begin{aligned}
				\sgn\big(\partial_{\eps_\w}\eps_\psi(\eps_{\w_1}; x_T)\big) &= \sgn((x_T-x_{T_2})(\eps_{\w_1}-\eps_{\w_2})(\eps_{\w_1}-\eps_{\w}^*)),\\
				\sgn(\partial_{\eps_\w} \eps_\psi(\eps_{\w_2}; x_T)) &= \sgn((x_{T_2}-x_T)(\eps_{\w_2}-\eps_{\w_1})).
			\end{aligned}
		\end{equation}
		With these preparations, we discuss the claimed geometry of stability boundary for $\eps_\psi>0$ and the different cases.  
		
		\textbf{(I). The case $\mathbf{x_T<x_{T_1}}$ and $\mathbf{x_T\leq x_{T_2}}$}. This applies to both situations of the theorem statement and we have already inferred $c_2,c_0>0$. Suppose $x_T<x_{T_2}$. 
		From \eqref{e:sgneps12} it follows that $\eps_{\w_2}<0<\eps_{\w_1}$, which means that $\eps_{\w_1}$ lies at the unique intersection point of the graph of $\eps_\psi(\cdot;x_T)$ with the $\eps_\w$-axis. Hence, to prove the claim that Theorem \ref{Thm_eps_curve} holds true for all $\epsrd>0$, it suffices to show that the graph also intersects the positive $\eps_\psi$-axis or is positive with vertical asymptote. As argued above, this is ensured if $\partial_{\eps_\w}\eps_\psi(\eps_{\w_1}; x_T)<0$. Since $\eps_{\w_1}>\eps_{\w_2}$ in the present situation, this sign is indeed negative for $\eps_\w^*<\eps_{\w_1}$ due to \eqref{e:sgnpartials}. In the case $\eps_\w^*>\eps_{\w_1}$ the sign is positive, but the graph is positive for $\eps_{\w_1}<\eps_{\w}<\eps_{\w}^*$ with a vertical asymptote at $\eps_\w^*$; compare Fig.~\ref{eps_curve_m_xr_n03}. For $x_T=x_{T_2}$ we already found above that the graph is linear with positive slope and root at $\eps_\w=\eps_{\w_1}>0$.
		
		\textbf{(II). The case $\mathbf{x_{T_{2}}<x_{T_{1}}}$.} We first note that the subcase $x_T\geq \xTs$ in the theorem statement certainly occurs since $\xTs>x_{T_{2}}$ follows from $x_{T_{2}}<x_{T_{1}}$ by direct calculation and using \eqref{e:signcond}. \\
		Suppose $x_T>\xTs$. Then, $\eps_{\w_1}>\eps_{\w_2}>0$ holds due to \eqref{e:sgneps12} and $\eps_{\w_1}=\eps_{\w_2}$ at $x_T=\xTs$. Therefore, \eqref{e:sgnpartials} implies $\partial_{\eps_\w} \eps_\psi(\eps_{\w_2})>0$. Due to the sign of the slope and $\eps_\w^*>\eps_{\w_2}$, the function $\eps_\psi(\cdot;x_T)$ takes negative values for $\eps_\w<\eps_{\w_2}$. However, $x_T>x_{T_2}$ implies $c_2<0$ for $\eps_\w>\eps_{\w_2}$, which means stabilization is possible for $\eps_\w<\eps_{\w_2}$ only. Thus, as claimed, stabilization of the straight motion by $\eps_\psi> 0$ is not possible for $x_T>\xTs$. \\
		Suppose now $x_T=\xTs$. Here $\eps_{\w_1}=\eps_{\w_2}<\eps_{\w}^*$ so that $\partial_{\eps_\w} \eps_\psi(\eps_{\w_1})=0$ and $a>0$ due to \eqref{e:sgneps12}. In the relevant range $\eps_\w<\eps_{\w_2}$, the function $\eps_\psi(\cdot;x_T)$ again takes negative values and hence, stabilization is not possible; see Fig.~\ref{cases} (b).\\
		Since $x_T=x_{T_2}$ was discussed in (I) above, the last subcase is $x_{T_{2}}<x_T<\xTs$. Analogous to before, the signs and monotonicity relations \eqref{e:sgneps12} and \eqref{e:rel2} imply $0<\eps_{\w_1}<\eps_{\w_2}<\eps_{\w}^*$. It follows that $\eps_\psi(\cdot;x_T)$ is smooth for $\eps_\w\leq \eps_{\w_2}$, positive for $\eps_{\w_1}< \eps_\w< \eps_{\w_2}$ and negative for $\eps_\w< \eps_{\w_1}$. Its parabolic shape follows from the quadratic numerator of $\eps_\psi(\cdot;x_T)$.
		
		\textbf{(III). The case $\mathbf{x_{T_{1}}{\leq}x_{T_{2}}}$.} As shown above, stabilization is possible if $x_T<x_{T_1}$, but is impossible for $x_T>x_{T_1}$. The marginal case $x_T=x_{T_1}$ implies $c_0=0$, i.e., a zero eigenvalue, and at $\eps_\psi=0$ we obtain $c_1=- \cZ<0$, which implies instability. Therefore, stabilization in the sense of the theorem statement is not possible.}
	
	\medskip
	{\black We come to the degenerate situation \eqref{e:degen}, which needs $x_T\neq x_{T_{2}}$. Since \eqref{e:sgneps12} and \eqref{e:rel2} hold, $\eps_{\w_1}=\eps_\w^*$ requires $x_T< x_{T_1}$ and $x_T< x_{T_{2}}$ so that $a<0$ and $\eps_{\w_2}\leq0<\eps_{\w_1}$. Then the linear branch of \eqref{e:degen} has negative slope $a$ and is non-positive at $\eps_\w=0$. Hence, it lies outside the relevant range and the vertical branch is the stability boundary.}
	
	{\black It remains to prove non-zero derivatives of the real part of the critical eigenvalue(s) with respect to the control parameters $\eps_\psi=\eps_\psi^0, \eps_\w=\eps_\w^0$ on the stability boundary. 
		In the case $\eps_\psi^0=0$, a transverse crossing goes along the $\eps_\w$-axis. Here \eqref{e:charpoly} reduces to $\lambda^2+c_2\lambda+c_1=0$ so that $\partial_{\eps_\w} \lambda = -\frac{\partial_{\eps_\w}c_1}{c_2} = -\frac{\ctA-\ctB x_T}{c_2}$,  
		which is negative since $x_T<x_{T_1}$.\\
		In the case $\eps_\psi^0>0$, the critical eigenvalues $\pm\rmi\sqrt{c_1}$ have common real part $\mu$ for nearby $\eps_\w, \eps_\psi$. Upon implicitly differentiating \eqref{e:charpoly}, a direct  computation for $\eps\in\{\eps_\w, \eps_\psi\}$ yields
		\begin{equation}\label{e:muprime}
			\partial_{\eps} \mu(\eps_\w^0, \eps_\psi^0) = -\frac{\partial_{\eps} (c_1 c_2 - c_0)}{2 (c_1 + {c_2}^2)}(\eps_\w^0, \eps_\psi^0),
		\end{equation}
		which is well-defined since $c_1,c_2>0$ on the stability boundary. In the case \eqref{e:degen}, $\eps=\eps_\w$ for a transverse crossing and \eqref{e:degen} gives $\partial_{\eps} (c_1 c_2 - c_0)=\eps_\psi-a(\eps_\w-\eps_{\w_2})$. This is positive since we found above $a<0$, and in this case, $\eps_{\w_1}>\eps_{\w_2}$ holds. Otherwise, the boundary is in a smooth branch of $\eps_\psi(\eps_\w;x_T)$ and from its definition we have  
		\[
		\eps_\psi'(\eps_\w^0;x_T)= - \partial_{\eps_\w} (c_1 c_2 - c_0)/b,\quad b:=\partial_{\eps_\psi} (c_1 c_2 - c_0)(\eps_\w^0, \eps_\psi^0).
		\]
		Due to the fact that $b\neq 0$, \eqref{e:muprime} is non-zero for $\eps=\eps_\psi$. For  $\eps=\eps_\w$, the numerator can be written as $-b \eps_\psi'(\eps_\w^0;x_T)$, which vanishes on the stability boundary precisely at the local maximum when this is parabolic, i.e., for $x_{T_{2}}<x_T<\xTs$.}
\end{proof}

{\black The study of linear stability is completed and now we move to the nonlinear analysis.}

\subsection{Bifurcation Analysis}
\label{NonLinearANalysis}
In this section we analyze the nonlinear effects of the stabilizing control based on the linear stability analysis of the previous section. 
In order facilitate the bifurcation analysis, we first shift the straight motion equilibrium point $(u_0,0,0,0)$ of \eqref{e:4DexplicitODE} to the origin by writing the surge variable as $u = u_0 + \tilde{u}$. In terms of $(\tilde{u},v,\w,\psi)$ we thus obtain
\begin{equation}
	\begin{pmatrix}
		\dot{\tilde{u}}\\
		\dot{{v}}\\
		\dot{{\w}}\\
		\dot{\psi}
	\end{pmatrix}=
	M^{-1}\begin{pmatrix}
		\mL{v}\w+{X}_H(\tilde{u})+{\tau(u_0+\tilde{u}})\cos{\eta}\\
		-\mL{(u_0+\tilde{u})\w}+{Y}_H(\tilde{u})+{\tau(u_0+\tilde{u})}\sin{\eta}\\
		{N}_H(\tilde{u})+x_T{\tau(u_0+\tilde{u})}\sin{\eta}\\
		r
	\end{pmatrix}.
	\label{ThrusterSys}
\end{equation} 
In the following we omit the tilde from $\tilde{u}$ to simplify the notation. 
We rewrite, expand in $u$ and, based on the results in \cite{SteinMacher2020}, already omit all cubic-order terms, which yield
\begin{equation}
	\label{System_Origin}
	\begin{pmatrix}
		\dot{u}\\
		\dot{{v}}\\
		\dot{{\w}}\\
		\dot{\psi}
	\end{pmatrix} =
	\begin{pmatrix}
		\tau_{11}(\cos\eta -1) + \kk_1 u + \kk_2 u^2 + \kk_3 v\w + [\tau_{12}{u}+\tau_{13}{u}^2]\cos{\eta}\\
		p_{22} v + \kk_5 \w + \kk_6 uv + \kk_7 u\w + f_1\left( v, \w \right) + [\tau_{21}+\tau_{22}{u}+\tau_{23}{u}^2]\sin{\eta}\\
		p_{32} v + \kk_9 \w + \kk_{10} uv + \kk_{11} u\w + f_2\left( v, \w \right) + [\tau_{31}+\tau_{32}{u}+\tau_{33}{u}^2]\sin{\eta}\\
		r
	\end{pmatrix}.
\end{equation}
{\black The coefficients $\kk_j$,$\tau_{ij}$ result directly from the expansion, but their explicit forms are not used in the following abstract analysis. The functions $f_1, f_2$ read}
\begin{equation*}
	f_1\left( v, \w \right) = a_{11}v\abs{v} + a_{12}v\abs{\w} + a_{21}\w\abs{v} + a_{22}\w\abs{\w},\quad
	f_2\left( v, \w \right) = b_{11}v\abs{v} + b_{12}v\abs{\w} + b_{21}\w\abs{v} + b_{22}\w\abs{\w},
\end{equation*}
with coefficients

\medskip
\begin{center}
	\begin{tabular}{ l l } 
		$a_{11} = D^{-1}\left( (I_z+m_{rr})Y_{\beta\abs{\beta}}-m_{vr}N_{\beta\abs{\beta}} \right)$, & $a_{12} = D^{-1} (I_z+m_{rr})Y_{\beta\abs{\gamma}}$,\\
		$a_{21} = D^{-1} (I_z+m_{rr})Y_{\abs{\beta}\gamma}$, & $a_{22} = D^{-1} \left( (I_z+m_{rr})Y_{\gamma\abs{\gamma}}-m_{vr}N_{\gamma\abs{\gamma}} \right)$,\\
		$b_{11} = D^{-1}\left( (m+m_{vv})N_{\beta\abs{\beta}} -m_{rv}Y_{\beta\abs{\beta}} \right)$, & $b_{12} = -D^{-1} m_{rv}Y_{\beta\abs{\gamma}}$,\\
		$b_{21} = -D^{-1} m_{rv}Y_{\abs{\beta}\gamma}$, & $b_{22} = D^{-1}\left( (m+m_{vv})N_{\gamma\abs{\gamma}} -m_{rv}Y_{\gamma\abs{\gamma}} \right)$.
	\end{tabular}
\end{center}
\medskip
Using again that our analysis focuses on the vicinity of the origin, we expand the functions $\cos\eta=1+\calO(\eta^2)$, $\sin\eta=\eta+\calO(\eta^3)$ and omit higher-order terms. Using the control form $\eta = \eps_\w\w+\eps_\psi\psi$ from \eqref{Pcontrol} this reduces \eqref{System_Origin} to
\begin{equation}
	\label{4D_Origin}
	\begin{pmatrix}
		\dot{u}\\
		\dot{{v}}\\
		\dot{{\w}}\\
		\dot{\psi}\\
	\end{pmatrix} =
	\begin{pmatrix}
		p_{11} u + U(u,v,\w)\\
		p_{22} v + p_{23}\w + p_{24}\psi + \kk_6 uv + (\kk_7+\tau_{22}\eps_\w) u\w + \tau_{22}\eps_\psi u\psi + f_1\left( v, \w \right)\\
		p_{32} v + p_{33}\w + p_{34}\psi + \kk_{10} uv + (\kk_{11}+\tau_{32}\eps_\w)u\w + \tau_{32}\eps_\psi u\psi + f_2\left( v, \w \right)\\
		r\\
	\end{pmatrix},
\end{equation}
where $U(u,v,\w)$ is a second-order nonlinear function{\black. In agreement with \S\ref{s:linstab}, the} coefficients are 
$p_{11}=\kk_1+\tau_{12}$, $p_{23}=\kk_5+\tau_{21}\eps_\w$, $p_{24}=\tau_{21}\eps_\psi$, $p_{33}=\kk_9+\tau_{31}\eps_\w$, $p_{34}=\tau_{31}\eps_\psi$, 
{\black i.e.,} the entries of the matrix $A$ from \eqref{Matrix_A}.
In particular, the linear part $A$ possesses a diagonal block structure with the eigenvalue $\lambda_1=p_{11}$ and a lower right $3\times 3$-submatrix block.

As a first step towards the nonlinear analysis, we discuss the simpler case of the pitchfork bifurcation, and then turn to the more involved Hopf bifurcation analysis.

\subsubsection{Pitchfork bifurcation for $\eps_\psi=0$}\label{s:pitch}
We keep $\eps_\psi=0$ fixed so that the last equation in \eqref{4D_Origin} can be dropped. Hence, the linear part $A$ reduces to its upper left $3\times 3$-submatrix that consists of a block $\lambda_1=p_{11}<0$ as well as {\black the} $2\times 2$-block  
{\black $B=(p_{ij})_{2\leq i,j\leq 3}$}. 
Theorem~\ref{Thm_eps_curve} implies that $B$ possesses an eigenvalue $\lambda<0$ as well as an eigenvalue that crosses through zero as $\eps_r$ crosses through $\eps_{r_1}$, which is a unique positive value for the HTC. 
In particular, $B_0:=B|_{\eps_r = \eps_{r_1}}$ has eigenvectors $e_0$ for the zero eigenvalue and $e_1$ for $\lambda$, and the bifurcation upon changing $\eps_r$ will be purely of steady states. We thus seek solutions of the reduced \eqref{4D_Origin} with zero left-hand side. Since $f_1, f_2$ are quadratic of second-order modulus type, we expect a nonsmooth pitchfork bifurcation as in the truncated normal form $\sigma x|x| + \teps x =0$. {\black Here $x$ is the unknown variable and} the sign of the parameter $\sigma\neq 0$ determines the super- or subcritical character of the bifurcation. 

The first steady state equation of the reduced \eqref{4D_Origin} can be solved for $u$ by the implicit function theorem since $\lambda_1=p_{11}< 0$ and $U$ is nonlinear. The resulting solution satisfies $u=\calO(v^2+r^2)$ so that substitution into the second and third equations contributes to a term of cubic-order.
In order to unfold the bifurcation, we write $\eps_r = \eps_{r_1} + \teps$ {\black so that} $B = B_0 + \teps B_1$ {\black with a matrix} $B_1$ {\black that} has vanishing first column. Next we choose the eigenvectors $e_j^*$ of the adjoint $B_0^\intercal$ such that $\langle e_j, e_{1-j}^*\rangle =0$ and $\langle e_j, e_j^*\rangle =1$, $j=1,2$. Changing coordinates $(v,r) = x e_0 + y e_1$, we project the second and third equations onto $\mathrm{span}(e_1)$, which results in 
\begin{align*}
	0&=\langle (B_0+ \teps B_1)(xe_0+ ye_1),e_1^*\rangle + \langle f(v,r),e_1^*\rangle \\
	&= \lambda y + \teps (x \langle B_1 e_0,e_1^*\rangle + y\langle B_1 e_1,e_1^*\rangle) + \langle f(v,r),e_1^*\rangle.
\end{align*}
Again, we may solve by the implicit function theorem since $\lambda< 0$, which yields $y=y(\teps,x) = \calO(|\teps x| + x^2)$. It remains to solve the projection onto $\mathrm{span}(e_0)$, which is given by 
\begin{align*}
	0&=\langle (B_0+ \teps B_1)(x e_0+ ye_1),e_0^*\rangle  + \langle f(v,r),e_0^*\rangle \\
	&= \teps (x \langle B_1 e_0,e_0^*\rangle + y(\teps,x)\langle B_1 e_1,e_0^*\rangle) + \langle f(v,r),e_0^*\rangle\\
	& = \teps x \langle B_1 e_0,e_0^*\rangle+ x|x|\langle f(e_0),e_0^*\rangle + \calO(3),
\end{align*}
where $\calO(3)$ is of cubic-order in $\teps, x$, and where we used that $f$ is of second-order modulus form. Thus, the truncated bifurcation equation reads 
\begin{equation}\label{e:bifpitch}
	\teps x \langle B_1 e_0,e_0^*\rangle+ x|x|\langle f(e_0),e_0^*\rangle=0.
\end{equation}
{\black  This gives the normal form since $\langle B_1 e_0,e_0^*\rangle$ coincides with the derivative with respect to the eigenvalue that is non-zero by 
	Theorem~\ref{Thm_xT}.} 
Numerical evaluation for the default HTC values yields the negative coefficients
\begin{equation*}
	\langle B_1 e_0,e_0^*\rangle \approx -4.04\cdot 10^{-2}, \quad
	\langle f(e_0),e_0^*\rangle \approx -1.32\cdot 10^{-3}.
\end{equation*}
Therefore, a nonsmooth pitchfork bifurcation occurs and it is supercritical since the steady state is stable for $\teps>0$. This means that a branch of stable steady motions emerges when decreasing $\eps_r$ from $\eps_{r_1}$. Indeed, we numerically find this as plotted in Fig.~\ref{pitchfork_epsr_v}.
\begin{figure}
	\centering
	\begin{tabular}{ccc}
		\includegraphics[width=0.29\linewidth]{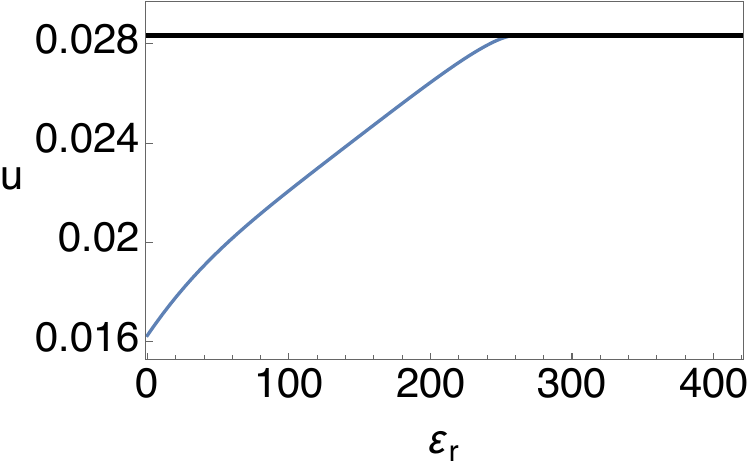} &
		\includegraphics[width=0.29\linewidth]{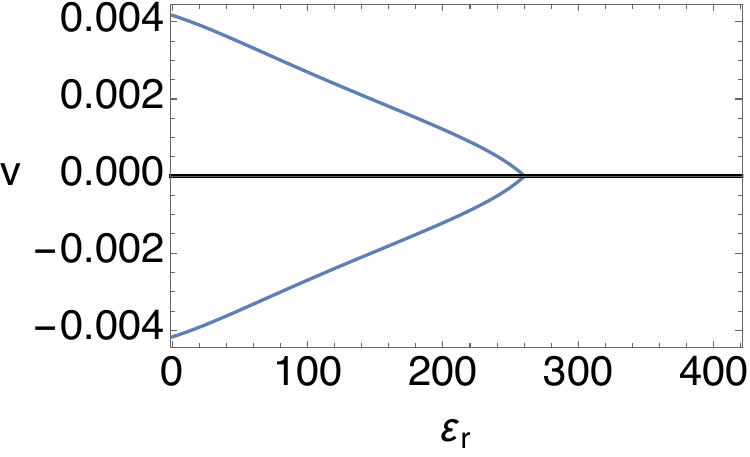} &
		\includegraphics[width=0.29\linewidth]{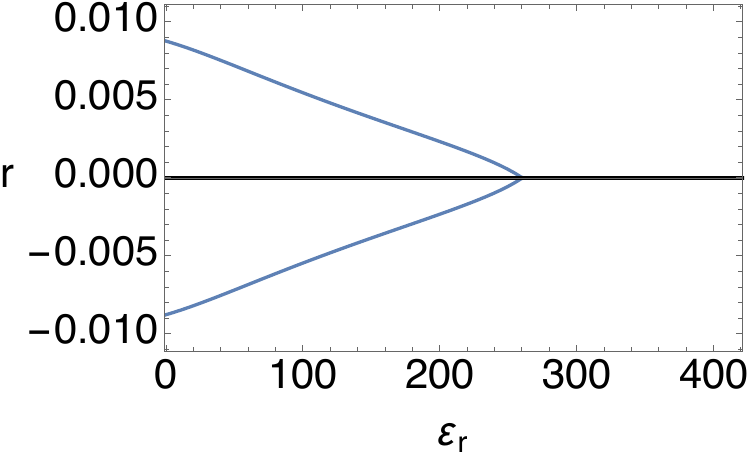}
	\end{tabular}
	\caption{The nonsmooth pitchfork bifurcation diagram of equilibrium points (blue) from the straight motion (black) in the 3D reduced system for $\eps_\psi=0$ computed by numerical continuation.}
	\label{pitchfork_epsr_v}
\end{figure}

{\black Within the 4D system \eqref{4D_Origin}, this pitchfork bifurcation manifests as follows. 
	We recall from  \S\ref{Stability_s_m} that for any $\eps_r\geq 0, \eps_\psi> 0$, the straight motion $(u_0,0,0,0)$ appears as the unique equilibrium point with positive $u=u_0$ solving \eqref{u_eq}. 
	However, for $\eps_\psi=0$ 
	a line of equilibria appears, given by $(u_0,0,0,\psi)$ with arbitrary $\psi$. 
	
	On the one hand, let $r_+>0>r_-$ be the $r$-components of the bifurcating equilibria in the 3D system. Then, due to $\dot\psi=r$, the $\psi$-components of the corresponding solutions in the 4D system with initial $\psi(0) = \psi_\pm(0)$ are, respectively, 
	\begin{equation}\label{e:bifpsi}
		\psi_\pm(t) := r_\pm t + \psi_\pm(0).
	\end{equation} 
	Near the bifurcation, $r_\pm\approx 0$ and $r_+>0>r_-$ hold, so that these solutions slowly drift unboundedly parallel to the line of equilibria in opposite directions. Since $\psi$ is an angular variable in the model, for large variations of $\psi$, the 4D phase space should be viewed as a cylinder. This is consistent with the equations due to the fact that for $\eps_\psi=0$ the vector field is independent of $\psi$. On this cylinder, the bifurcating solution are periodic orbits with winding number $\pm1$ for $r=r_\pm$, respectively.
	
	On the other hand, at $\eps_\psi=0$ the linear part $A$ of \eqref{4D_Origin} in \eqref{Matrix_A} possesses a zero eigenvalue for any $\eps_r$ and a double zero eigenvalue at $\eps_r=\eps_{r_1}$ with a $2\times 2$-Jordan block. The remaining two eigenvalues are negative. 
	The occurrence of the Jordan block suggests that the unfolding contains elements of a symmetric Bogdanov--Takens bifurcation beyond the pitchfork; cf.\ \cite{Carr}. Indeed, the stability boundary \eqref{Thm_eps_curve} is a curve of Hopf bifurcations in parameter space that will be further studied in the next section. We numerically find global bifurcations in \S\ref{Numerical_Bif_An}, but its rigorous analysis is beyond the scope of this paper. We note that the situation is degenerate by the line of equilibria and the nonsmooth nonlinearity, which already makes the following Hopf bifurcation analysis more involved. 
}

\subsubsection{Hopf bifurcations} \label{s:hopf}

{\black Consider any control parameters $\eps_r\geq0$, $\eps_\psi>0$ on the stability boundary defined by \eqref{function_eps}. 
	Due to Theorem~\ref{Thm_xT}, the eigenvalues of the linearization at $(u,v,r,\psi)=(u_0,0,0,0)$ consist of one complex conjugate pair with non-zero imaginary parts, and two negative real eigenvalues.} 
Thus, near any such $(\eps_r,\eps_\psi)$, the linear part $A$ of \eqref{4D_Origin} can be transformed into a real normal form as a $4\times 4$-block-diagonal matrix. {\black The blocks from top left to bottom right can be arranged as} $\lambda_1=p_{11}<0$ from \eqref{4D_Origin}, a $2\times 2$-matrix for the complex eigenvalues $\lambda_\pm=\mu\pm \rmi\omega$ with $\omega> 0$, and a real eigenvalue $\lambda_4<0$. {\black Due to Theorem~\ref{Thm_xT}, $\mu$ strictly decreases as $\eps_\w,\eps_\psi$ cross the stability boundary into the stable region.}

In the following, we perform the coordinate changes and identification of terms that allow to apply the theory from \cite{SteinMacher2020}. This will then justify to neglect the terms we drop in the coming steps. 
In particular, since $\lambda_1<0$ is already a block of the linear part $A$ from \eqref{Matrix_A}, the results from \cite{SteinMacher2020} imply  that we can neglect the first equation of  \eqref{4D_Origin} for our bifurcation analysis. Therefore, we subsequently analyze the lower right $3\times 3$-matrix $(v,r,\psi)$, which contains the linearly oscillating part. 
We define the matrix $\bT=(\ba|\bb|\bs)$ with columns $\ba, \bb, \bs\in\R^3$ from the eigenvectors $\zeta_+=\ba+i\bb$, $\zeta_-=\ba-i\bb$, $\bs$, of the eigenvalues $\mu+i\omega$, $\mu-i\omega$, $\lambda_4$, respectively. Notably, $\mu=0$ along the curve $\eps_\psi(\eps_\w{\black; x_T})$ {\black from \eqref{e:epspsixT}}. 

Setting $\xi:=(\xi_1,\xi_2,\xi_3)^\intercal=\bT^{-1}(v,r,\psi)^\intercal$, the $(v,r,\psi)$-subsystem of \eqref{4D_Origin} takes the form
\begin{equation}
	\label{3D_transformed}
	\begin{pmatrix}
		\dot{\xi_1}\\
		\dot{\xi_2}\\
		\dot{\xi_3}\\
	\end{pmatrix} =
	\begin{pmatrix}
		\mu		&-\omega	&0 \\
		\omega	&\mu	&0\\
		0		&0		&\lambda_4
	\end{pmatrix}
	\begin{pmatrix}
		\xi_1\\
		\xi_2\\
		\xi_3
	\end{pmatrix}+h_2(\xi)+\calR,
\end{equation}
where $h_2(\xi)$ contains all relevant quadratic-order terms. {\black These have the form}
\begin{equation*}h_2(\xi)=\bT^{-1}
	\begin{pmatrix}
		f_1(\bT\cdot\xi)\\
		f_2(\bT\cdot\xi)\\
		0\\
	\end{pmatrix}=
	\begin{pmatrix}
		\T_{11}g_1(\xi) + \T_{12}g_2(\xi)\\
		\T_{21}g_1(\xi) + \T_{22}g_2(\xi)\\
		\T_{31}g_1(\xi) + \T_{32}g_2(\xi)\\
	\end{pmatrix},
\end{equation*}
with quadratic-order functions $g_1, g_2$ discussed below. {\black Relevant will be the coefficients } 
\begin{align*} 
	\T_{11}&=\frac{b_2s_3-b_3s_2}{\det(\bT)}, &\T_{12}&=\frac{-b_1s_3+b_3s_1}{\det(\bT)}, 
	&\T_{21}&=\frac{-a_2s_3+a_3s_2}{\det(\bT)}, \quad &\T_{22}&=\frac{a_1s_3-a_3s_1}{\det(\bT)},
\end{align*} 
{\black where} $a_j, b_j, s_j, j\in\{1,2,3\}$ {\black are} the components of the vectors $\ba, \bb, \bs$, respectively.

All missing terms are collected in $\calR$, including the nonlinear terms involving $u$, i.e., $uv, ur, u\psi$, from \eqref{4D_Origin}, which turn out to be irrelevant at leading-order due to \cite[\S2, \S4]{SteinMacher2020}. Furthermore, \cite[Thm 2.3]{SteinMacher2020} implies that the third component of the transformed system \eqref{3D_transformed}, $\xi_3$, will belong to higher-order terms 
as well. {\black With these preparations, and} using the shorthand $[[\cdot]]:=\cdot|\cdot|$ {\black for the second-order modulus terms}, the relevant functions $g_1, g_2$ read 
\begin{align*}
	g_1(\xi_1,\xi_2) =& a_{11}[[a_1\xi_1+b_1\xi_2]] + a_{12}(a_1\xi_1+b_1\xi_2)\abs{a_2\xi_1+b_2\xi_2} \\
	&+ a_{21}(a_2\xi_1+b_2\xi_2)\abs{a_1\xi_1+b_1\xi_2} + a_{22}[[a_2\xi_1+b_2\xi_2]],\\
	g_2(\xi_1,\xi_2) =& b_{11}[[a_1\xi_1+b_1\xi_2]] + b_{12}(a_1\xi_1+b_1\xi_2)\abs{a_2\xi_1+b_2\xi_2} \\
	&+ b_{21}(a_2\xi_1+b_2\xi_2)\abs{a_1\xi_1+b_1\xi_2} + b_{22}[[a_2\xi_1+b_2\xi_2]].
\end{align*}
{\black At the linear and quadratic-order, the first two equations of \eqref{3D_transformed} are independent of $\xi_3$ such that only these two equations are relevant for the leading-order bifurcation analysis}. 
Setting $(\xi_1,\xi_2)=(\rp\cos\varphi,\rp\sin\varphi)$, these become
\begin{align}
	\begin{cases}
		\dot{\rp} &= \mu \rp+\chi(\varphi) \rp^2 + \calO(\rp^3),\\
		\dot{\varphi} &= \omega + \Omega(\varphi)\rp + \calO(\rp^2).
	\end{cases}
	\label{polar_coord}
\end{align}  
{\black Decisive for the bifurcation is the expression $\chi(\varphi)$, which can be determined using \cite[Thm 2.3]{SteinMacher2020}. With the abbreviations} $c:=\cos\varphi, s:=\sin\varphi$ and $\RR:=c\T_{11}+s\T_{21}, \SS:=c\T_{12}+s\T_{22}$ {\black this gives} 
\begin{gather}
	\begin{aligned}\label{eq:chi}
		\chi(\varphi) =& (a_{11}\RR+b_{11}\SS)[[a_1c+b_1s]] + (a_{12}\RR+b_{12}\SS)(a_1c+b_1s)\abs{a_2c+b_2s} \\
		&+ (a_{21}\RR+b_{21}\SS)(a_2c+b_2s)\abs{a_1c+b_1s} + (a_{22}\RR+b_{22}\SS)[[a_2c+b_2s]].
	\end{aligned}
\end{gather}

Before formulating the theorem, we recall the two generic forms of a Hopf bifurcation: in the supercritical (or safe) case the stable fixed point {\black becomes unstable} when the stable periodic orbit is created{\black, so that the stable periodic orbit coexist with the unstable equilibrium. I}n the subcritical (or unsafe) scenario the unstable fixed point becomes stable when an unstable periodic orbit is born. {\black We next show that this criticality is determined by a non-zero sign of 
	\[
	\Sigma:=\int_0^{2\pi}\chi(\varphi)\dd\varphi.
	\]
	Here $\chi(\varphi)$ depends on the chosen $\eps_\w,\eps_\psi$ on the stability boundary \eqref{e:epspsixT}.}

\begin{theorem}\label{t:crit}
	{\black Let $\eps_\psi>0$ and $x_T<\xTs$, if $x_{T_2}<x_{T_1}$, or $x_T<x_{T_1}$, if $x_{T_1}<x_{T_2}$; see \eqref{e:xTthresholds}. 
		Then, as the control parameter values cross the stability boundary \eqref{function_eps}, 
		the straight motion equilibrium} undergoes a Hopf bifurcation which is supercritical if 
	{\black $\Sigma<0$} 
	and subcritical {\black if $\Sigma>0$}. Furthermore, {\black in terms of the real part $\mu$ of the critical eigenvalues,} the amplitude of the periodic orbit is given by
	\begin{equation}\label{e:amp}
		\rp = -\frac{2\pi}{\Sigma} \mu + \calO\left(\mu^2\right).
	\end{equation}
\end{theorem}

{\black \begin{remark}\label{r:windHopf}
		The  $\psi$-component of the bifurcating periodic orbits lies in an order $\mu$ neighborhood of the equilibrium value $\psi=0$. Therefore, these solutions have winding number zero in the 4D phase space viewed as a cylinder with $\psi$ from the unit circle. 
\end{remark}}

\begin{proof}
	We recall from Theorem  
	{\black \ref{Thm_xT} that the real part $\mu$ of a simple} complex conjugate pair of eigenvalues crosses the imaginary axis {\black with non-zero derivative as} $(\eps_\w,\eps_\psi)$ crosses the {\black graph} of \eqref{function_eps}. Thus, the equivalent system \eqref{4D_Origin} is amenable to \cite[Cor. 4.7]{SteinMacher2020}. In particular,  the non-oscillatory linear part is invertible since $\lambda_1,\lambda_4< 0$ on the stability boundary by Theorems~\ref{Thm_eps_curve}  {\black and \ref{Thm_xT}}. 
	Due to \cite[Cor. 4.7]{SteinMacher2020}, the criticality of the Hopf bifurcation is that of \eqref{polar_coord}, and determined by the sign of {\black $\Sigma$}. Finally, the directly related claimed leading-order amplitude of the bifurcating periodic orbits follows from \cite[Prop. 3.8]{SteinMacher2020}.
\end{proof}

{\black The quadratic nonsmooth nature of the nonlinear terms is reflected in the leading-order linear dependence of the amplitude $r$ in \eqref{e:amp}, which would be a square root in the smooth case. This difference to the smooth situation appears also in the unfolding of the pitchfork bifurcation in \S\ref{s:pitch} with respect to $\eps_\w$. Similar to the smooth case, concrete computations of $\Sigma$ are complicated by the fact that the eigenvectors enter non-trivially. Hence}, even with the formula \eqref{eq:chi} and despite {\black the fact} that all terms in $\chi$ can be explicitly integrated, it appears difficult to determine the sign of $\Sigma$ analytically. 

Nevertheless, numerical evaluation of all these quantities and thus of $\Sigma$ is highly accurate and almost instantaneous on modern computers.
This makes it possible to readily predict the criticality as well as the leading-order expansion of the bifurcating periodic solutions. 
{\black 
	Rather than $\Sigma$, we compute the leading-order dependence of the amplitude in terms of the control parameter $\eps_\w$. This gives quantitative insight into the sensitivity of the resulting ship motion in addition to information on the criticality of the bifurcation. Let $(\eps_\w^0, \eps_\psi^0)$ denote some chosen control parameters on the stability boundary and set $\tilde\eps_\w:=\eps_\w-\eps_\w^0$. Expanding \eqref{e:amp} in terms of $\tilde\eps_\w$ yields 
	\begin{equation}\label{e:ampmu}
		\rp = -\ampc \tilde\eps_\w + \calO\left(\tilde\eps^2\right), \quad \ampc:=\frac{2\pi}{\Sigma}\partial_{\eps_\w} \mu(\eps_\w^0, \eps_\psi^0),
	\end{equation}
	which can be evaluated using \eqref{e:muprime}. 
	By Theorem~\ref{Thm_xT} we know that $\partial_{\eps_\w} \mu(\eps_\w^0, \eps_\psi^0)<0$ holds, except to the right of the maximum of the stability boundary in case it has parabolic shape. Hence, $\ampc$ has the opposite sign of $\Sigma$ except in the latter situation. In Fig.~\ref{Sigma_epsR_2pi_xTn03_p01} (a) we plot the results for the classical HTC, where the stability boundary is monotone. Since $b_r>0$, the Hopf bifurcation is supercritical for any stabilizing P-control, which means a safe control scenario. The dependence of $b_r$ on $\eps_\w$ is non-monotone in the interior, but grows significantly and monotone near the right boundary, where $\eps_\psi$ tends to zero. As a result, the amplitude of the bifurcating stable periodic orbits is more sensitive to $\eps_\w$ in this region of the stability boundary. 
}

Concerning {\black changes in} the thruster position {\black from the HTC default}, we recall from \S\ref{s:linstab} that there are four {\black controllable} cases as shown in Fig.~\ref{sketch_xT2}. At $x_T=-0.3$, which is case 2, {\black stabilization by $\eps_\psi>0$} is possible for approximately $\eps_\w\in(41.9,304.9)$; 
{\black at $\eps_\psi=0$ any $\eps_\w>305$ stabilizes the straight motion equilibrium; see Fig.~\ref{eps_curve_m_xr_n03} (a)}. We plot the resulting values of {\black $\ampc$} in Fig.~\ref{Sigma_epsR_2pi_xTn03_p01} (b). {\black Here $b_r$ increases monotonically in $\eps_\w$, again with stronger growth near the right boundary, where $\eps_\psi$ tends to zero.} 
{\black For $x_T=-0.16$, which corresponds to case 3 in Fig.~\ref{sketch_xT2}, we omit the plot in which we find a qualitatively reflected $\ampc$ graph}{\black, so that again the sensitivity is large for small values of $\eps_\psi$ on the stability boundary. The results of the last case 4 with $x_T=0.16$ are plotted in Fig.~\ref{Sigma_epsR_2pi_xTn03_p01} (c). In this case, the interval $\eps_\w\in(514,569.9)$ of $\eps_\w$-values that admit stabilization by $\eps_\psi$ is bounded by the zeros of the parabolic shaped stability boundary; see Fig.~\ref{cases} (c). 
	Its maximum leads to a sign change of $\ampc$, consistent with a supercritical bifurcation on all parts of the boundary. Here $\eps_\psi$ is small near both endpoints, but the sensitivity is larger near the left boundary.}

\begin{figure}
	\centering
	\begin{tabular}{ccc}
		\includegraphics[width=0.3\linewidth]{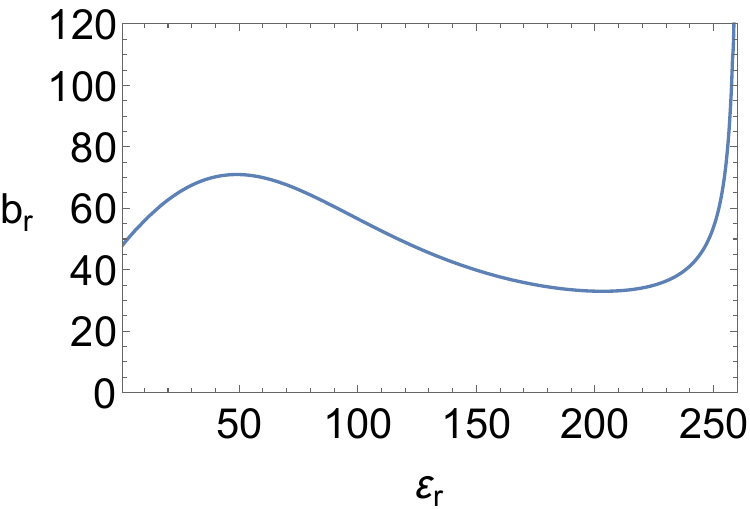} &
		\includegraphics[width=0.3\linewidth]{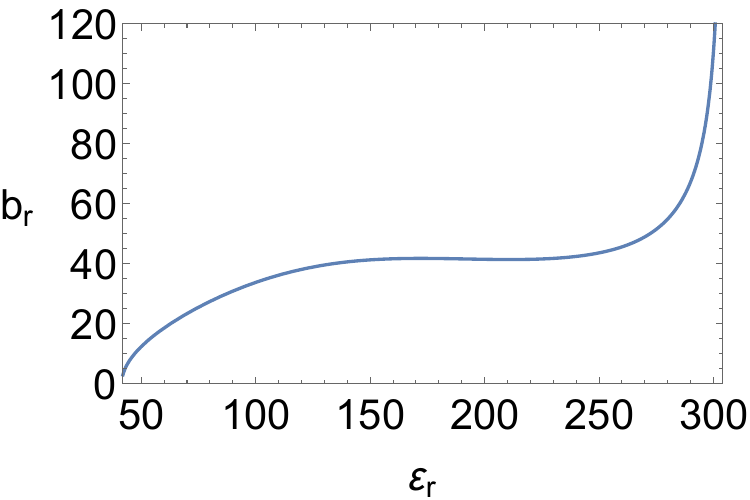} &
		\includegraphics[width=0.3\linewidth]{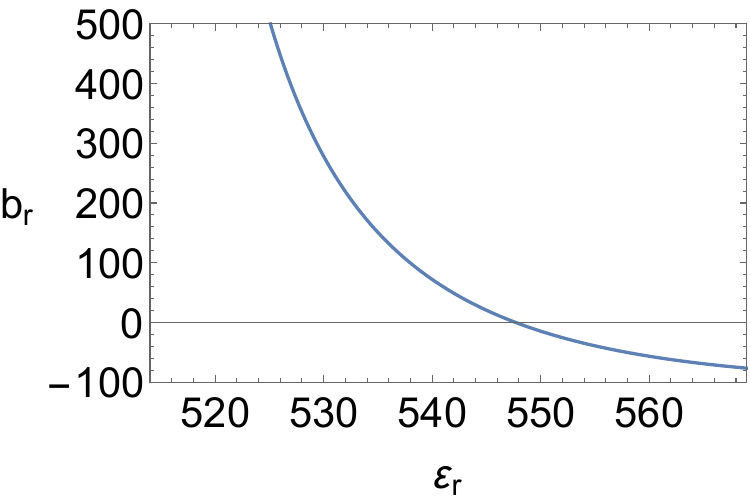}\\
		(a) & (b) & (c)
	\end{tabular}
	\caption{Values of $\ampc$ for {\rm(a)} $\eps_\w\in[1,260]$ in the HTC, see Fig.~\ref{sketch_xT2} case 1; {\rm(b)} $\eps_\w\in[42,304]$ with $x_T=-0.3$, see Fig.~\ref{sketch_xT2} case 2; and {\rm(c)} $\eps_\w\in[514,569]$ with $x_T=0.16$, case 4 in Fig.~\ref{sketch_xT2}. The graph in (c) is capped since the $\ampc$ values become very large.}
	\label{Sigma_epsR_2pi_xTn03_p01}
\end{figure}

\section{Numerical Bifurcation Analysis}\label{Numerical_Bif_An}
In this section we present numerical results that corroborate and illustrate the analysis of the previous sections based on implementing the model in the continuation software AUTO \cite{auto}. 
{\black In this way the stability boundary for the HTC can be computed by numerically tracking the Hopf bifurcation locus.
	We find that this agrees, up to numerical error, with the analytical prediction \eqref{function_eps} of Theorem~\ref{Thm_eps_curve}.}

\subsection{Continuation of Periodic Solutions}
We have employed numerical bifurcation and continuation to compute the periodic orbits that bifurcate from this stability boundary along curves in the $(\eps_\w, \eps_\psi)$-plane for $\eps_\psi>0$. 
{\black Confirming the} analytically predicted 
supercritical nature of the Hopf bifurcations, {\black these periodic orbits exist in a region below the stability boundary; see Fig. \ref{eps_curve_m}}. As shown in \S\ref{s:pitch}, along the $\eps_\w$-axis, i.e., for $\eps_\psi=0$, the Hopf bifurcation {\black is replaced by a supercritical} pitchfork bifurcation 
{\black in the reduced} 3D system for $u,v,r$. 

{\black Let $(u_\pm,v_\pm, r_\pm)\in\R^3$ denote the bifurcating equilibria in the 3D reduced system. In the 4D system \eqref{4D_Origin}, 
	these take the form $\bv_\pm(t)=(u_\pm,v_\pm,r_\pm,\psi_\pm(t))\in\R^4$ with $\psi_\pm$ from \eqref{e:bifpsi}. 
	In \S\ref{s:pitch} we have identified these as periodic orbits in the 4D phase space viewed as a cylinder at $\eps_\psi=0$.  
	However, for any $\eps_\psi>0$, the vector field depends linearly on $\psi$ through the P-control \eqref{Pcontrol}. This is generally inconsistent for a model with $\psi$ being the yaw angle, i.e., with the cylindrical geometry of the 4D phase space. This can be ignored as long as the solutions under consideration have $\psi$ contained in an interval of length less than $2\pi$. In particular, this is the case in the stability and Hopf bifurcation analysis, where the bifurcating solutions have winding number zero; see Remark~\ref{r:windHopf}. 
	Nevertheless, since $\psi_\pm(t)$ are unbounded this issue has to be addressed when studying the perturbation of $\bv_\pm$ for $\eps_\psi\approx0$. Therefore, for such a consideration the control law $\eta$ in \eqref{Pcontrol} needs to be made periodic in $\psi$. 
	
	\begin{remark}\label{r:persist}
		Consider a control that is $2\pi$-periodic in $\psi$, that equals \eqref{Pcontrol} at $\eps_\psi=0$, and that is smooth in $u,v,\w,\psi,\eps_\w,\eps_\psi$. Assume that the linearization of the 3D reduced system in $(u_\pm,v_\pm,r_\pm)$ has no eigenvalues on the imaginary axis. (We have numerically verified this for the HTC for all $\eps_\w\geq 0$.) 
		We claim that for $0<\eps_\psi\ll 1$ the corresponding periodic orbits $\bv_\pm$ are perturbed to nearby periodic orbits. Indeed, the linearization of \eqref{4DSystem} at each point on $\bv_\pm$ has one zero eigenvalue in the direction of the flow and the three eigenvalues of the 3D reduced system. Therefore, the orbit is a compact, normally hyperbolic invariant manifold without boundary, which is structurally stable under perturbations of a parameter.
	\end{remark}
}

For definiteness, we choose to replace the linear $\psi$-term {\black in the control law} by a sinus,
\begin{equation}\label{e:sinus}
	\eta= \eps_\w\w+\eps_\psi\sin(\psi).
\end{equation}
{\black This satisfies the assumptions of Remark~\ref{r:persist}. Since also $\sin'(0)=1$, we have that} the stability boundary, the supercritical nature of the bifurcations and {\black $\bv_\pm$} are identical to before. However, this choice {\black introduces the new straight motion equilibria $(u,v,r,\psi)=(u_0,0,0,\pm\pi)$, where the ship direction is exactly opposite to the a priori chosen reference direction $\psi=0$. This choice of $\eta$ also has an impact on the continuation in $\eps_\w$ and $\eps_\psi$ of the periodic orbits that bifurcate from the Hopf bifurcations on the stability boundary. We numerically found that there is 
	no qualitative change for $\eps_\psi\geq 1$. Here we have chosen $1$ as a relative small reference value and will discuss smaller values later.} Instead of the sinus we could choose a function that is the identity for $|\psi|$ below some threshold, but globally smooth and periodic. {\black This} would provide results that fully coincide with the P-control as long as $|\psi|$ is below the threshold.

For the sinusoidal control law, we find numerically the presence of a smooth surface of periodic solutions, parameterized by the control parameter values `under' the stability boundary curve within the range $\eps_r\geq 0, \eps_\psi\geq 1$. We have confirmed this along a number of different axis-aligned curves, including the $\eps_\psi$-axis, {\black where} $1\leq \eps_\psi<45.8$. 
In Fig.~\ref{bif_diag_epspsi_v} we plot different views of the bifurcation diagram {\black for $\eps_\w\approx 21.2$, chosen as an arbitrary non-zero value within the unstable region.} 
This is completely analogous for 
continuations along different curves. 
{\black I}n Fig.~\ref{periodic_orbits_epsr10} we plot some views  
of periodic solutions {\black on this branch.  
	Similar to the results in \cite{STC2007}, we observe a monotone} 
growth of the $\psi$-range as $\eps_\psi$ decreases.

\begin{figure}
	\centering
	\begin{tabular}{ccc}
		\includegraphics[width= 0.3\linewidth]{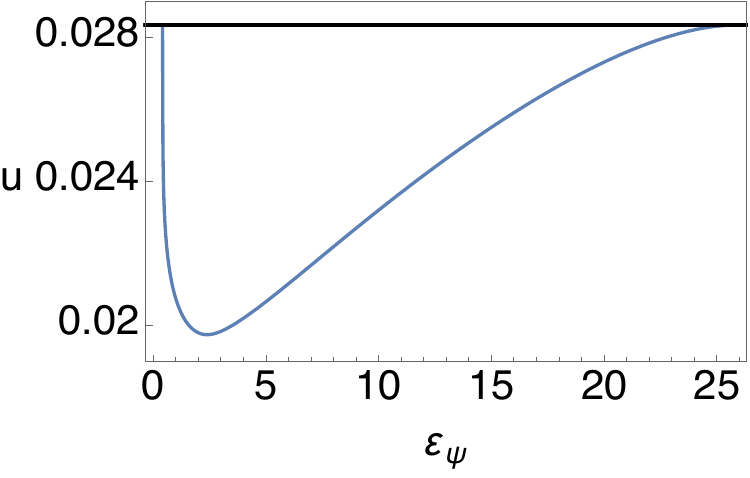} &
		\includegraphics[width= 0.3\linewidth]{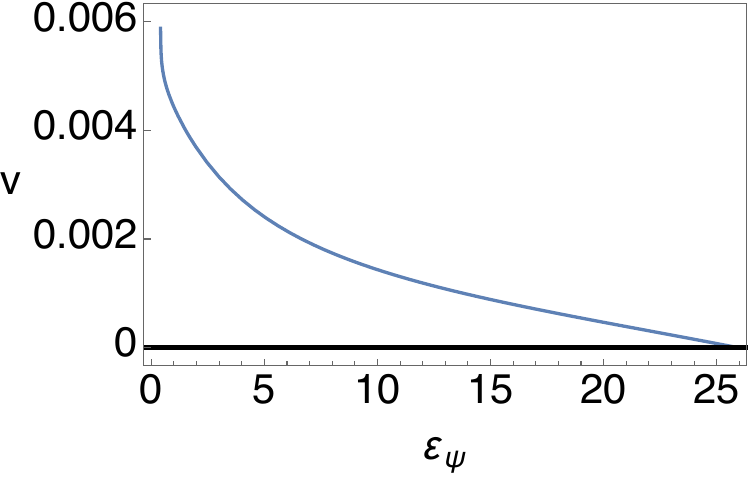} &
		\includegraphics[width= 0.3\linewidth]{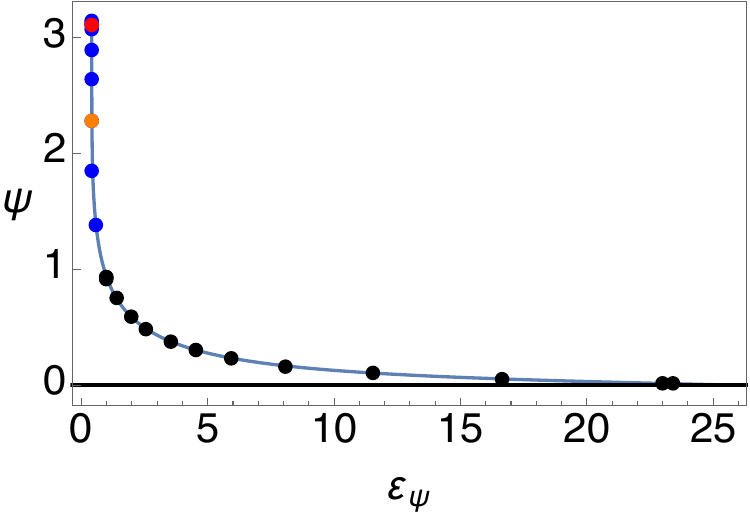}
	\end{tabular}
	\caption{{\black Different views of the global branch of periodic orbits emanating from the (nonsmooth) Hopf bifurcation at $\eps_\w\approx 21.2, \eps_\psi\approx 25.9$. At $\eps_\psi\approx 0.408$ the branch appears to terminate in a heteroclinic bifurcation. Plotted are the maximum values of $u,v,\psi$ of the periodic orbits; $\psi$ is 
			in radians. The bullets in the $\psi$-diagram mark the location of the solutions plotted in Fig.~\ref{periodic_orbits_epsr10}: the black bullets are solutions in Fig.~\ref{periodic_orbits_epsr10} (a,b), and the other colors match with the corresponding colored solutions in Fig.~\ref{periodic_orbits_epsr10} (c).}}
	\label{bif_diag_epspsi_v}
\end{figure}

\begin{figure}
	\centering
	\begin{tabular}{ccc}
		\includegraphics[height=0.3\linewidth]{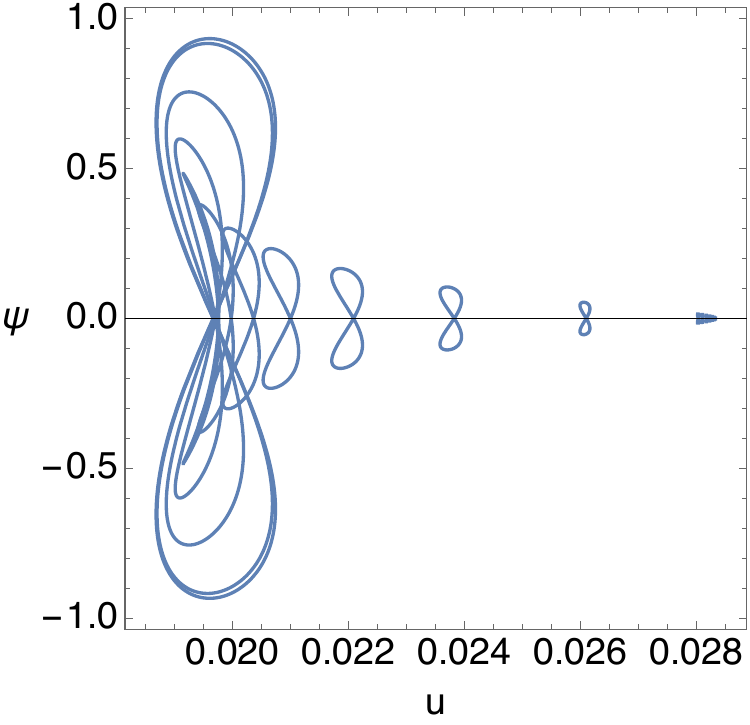}&
		\includegraphics[height=0.3\linewidth]{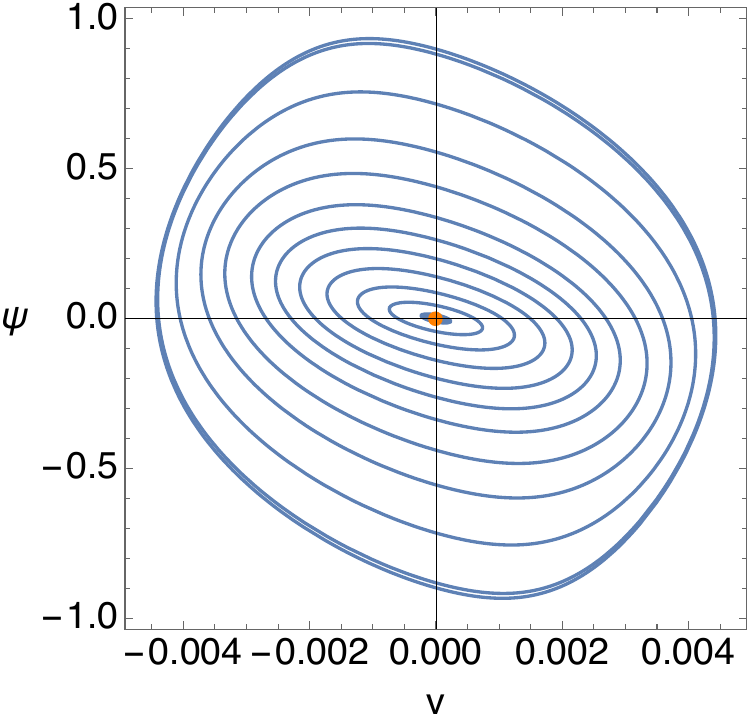} & 
		\includegraphics[height=0.3\linewidth]{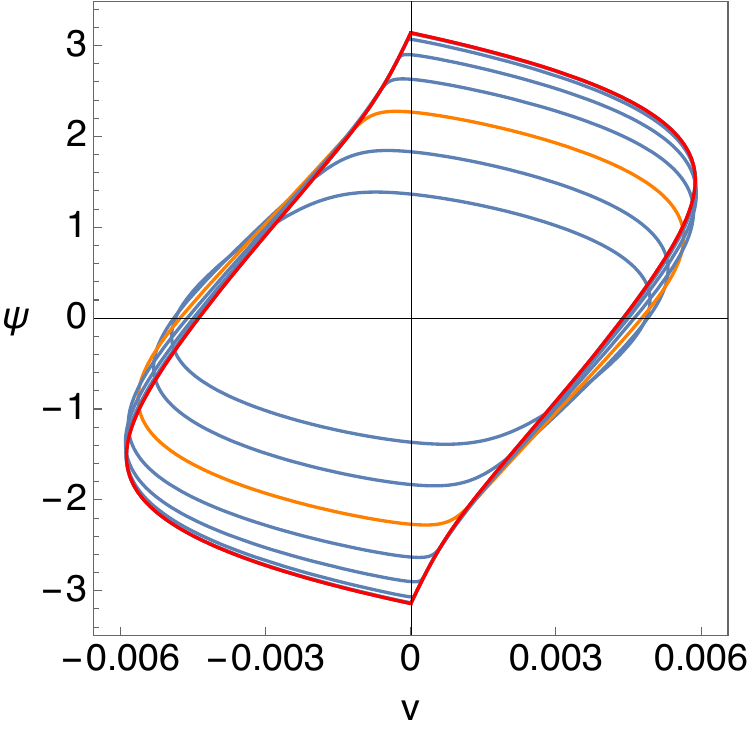}\\
		(a) & (b) & (c)
	\end{tabular}
	\caption{Sample profiles of periodic orbits from Fig.~\ref{bif_diag_epspsi_v}. In {\rm(a,b)} $\eps_\psi$ values are between $25.9$ at the Hopf bifurcation and $\eps_\psi=1$. In {\rm(c)} $\eps_\psi$ values are between $1$ and the heteroclinic bifurcation point $\eps_\psi\approx 0.408$ with period $T\approx 18570$. $\psi$ is given in radians.}
	\label{periodic_orbits_epsr10}
\end{figure}

\begin{figure}
	\centering
	\begin{tabular}{ccc}
		\includegraphics[height=0.2\linewidth]{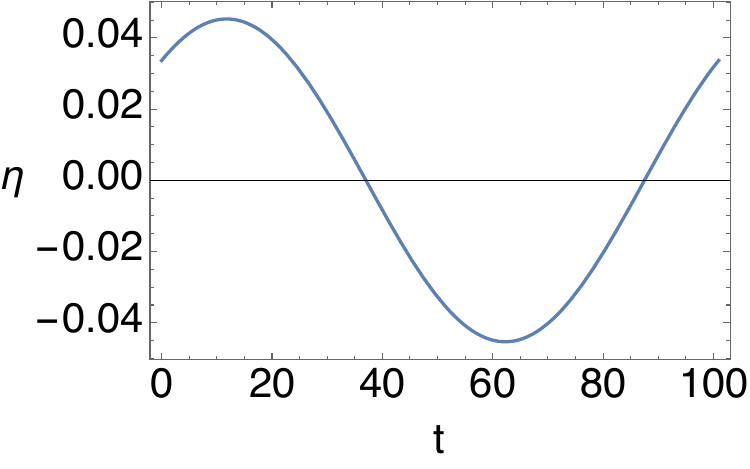} & 
		\includegraphics[height=0.2\linewidth]{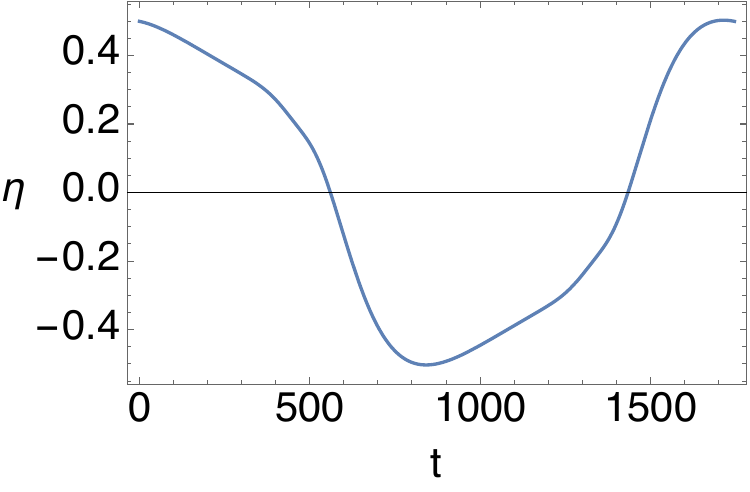} & 
		\includegraphics[height=0.2\linewidth]{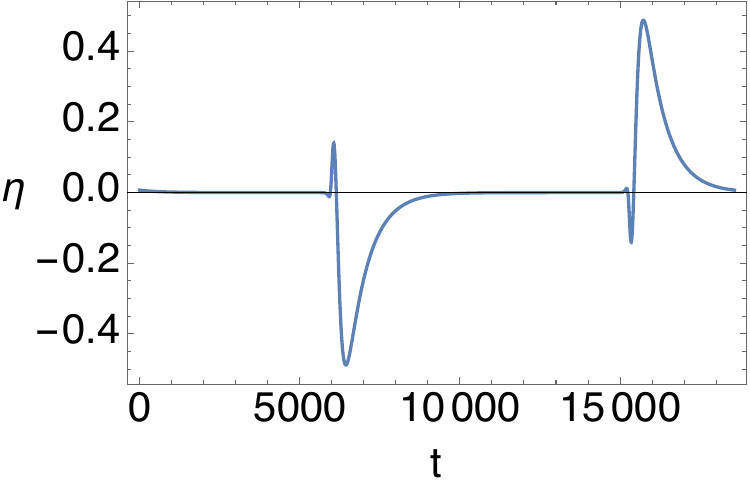}\\
		(a) & (b) & (c)
	\end{tabular}
	\caption{{\black Steering} angles ${\eta} = \eps_rr+\eps_\psi\sin(\psi)$ in {\black radians} over one period $t\in[0,T]$ of 
		{\black selected solutions from Fig.~\ref{periodic_orbits_epsr10}}: 
		{\rm(a)} $\eps_\psi\approx 25.6$ and $T\approx 101$ (near the origin of Fig.~\ref{periodic_orbits_epsr10} (b)); 
		{\rm(b)} $\eps_\psi\approx 0.419$ and $T\approx 1747$ (orange in Fig.~\ref{periodic_orbits_epsr10} (c)); 
		{\rm(c)} $\eps_\psi\approx 0.408$ and $T\approx 18570$ (red in Fig.~\ref{periodic_orbits_epsr10} (c)).
	}
	\label{periodic_orbit_etas}
\end{figure}

{\black On the one hand, decreasing $\eps_\psi$ further,} the branch appears to terminate for 
$\eps_\psi\approx 0.408$ in a heteroclinic bifurcation. 
{\black It seems that here} a heteroclinic cycle {\black exists, which consists of a symmetric pair of heteroclinic orbits} between 
{\black the additional} straight motion equilibrium points {\black with $\psi\approx\pm \pi$ \si{rad}. 
	Along each of the heteroclinic solutions, the ship direction completes a full circle and asymptotes to the direction opposite to the reference $\psi=0$. On one of these orbits the rotation is clockwise and on the other it is counter-clockwise.}
We show one view {\black of a periodic solution} near this cycle in Fig.~\ref{periodic_orbits_epsr10} (c) {\black as a red curve}.  

{\black On the other hand, increasing $\eps_\psi$ from zero, we observe that $\bv_\pm$ perturb to nearby periodic solutions for $0<\eps_\psi\ll 1$. This confirms the prediction from Remark~\ref{r:persist}. Further increasing $\eps_\psi$, 
	each of these solutions appears to terminate at $\eps_\psi\approx 0.408$ in one of the heteroclinic orbits that together form the aforementioned heteroclinic cycle. 
	In Fig.~\ref{periodic_orbits_epsr10} (c) this heteroclinic orbit is near one of the parts of the red trajectory with fixed sign of $v$. 
	
	This difference between decreasing and increasing $\eps_\psi$ can be understood from the winding numbers of the periodic orbits in the cylinder. The winding number is a homotopy invariant and thus must remain constant along smooth branches. The orbits $\bv_\pm$ that bifurcate from $\eps_\psi=0$ according to Remark~\ref{r:persist} have winding numbers $\pm 1$, respectively. Hence, the heteroclinic cycle as a whole has winding number zero, which is compatible with the zero winding number of the periodic orbits that bifurcate from the Hopf points.} 
It appears that this is the way in which the entire region `under' the stability boundary up to the $\eps_\w$-axis is organized.
{\black We suspect that the heteroclinic cycle and its unfolding can be rigorously studied by analyzing the Bogdanov--Takens-type point at $\eps_\psi=0$, $\eps_\w=\eps_{\w_1}$ in the globally cylindrical geometry. In terms of $x_T$ this unfolding might possess additional structure at $x_T=x_{T_-}$, where the stability boundary is vertical, and at $x_T=\xTs$, where the stability region has shrunk to a point; compare Fig.~\ref{sketch_xT2}.} 

\begin{figure}
	\centering
	\begin{tabular}{ccc}
		\includegraphics[height=0.13\linewidth]{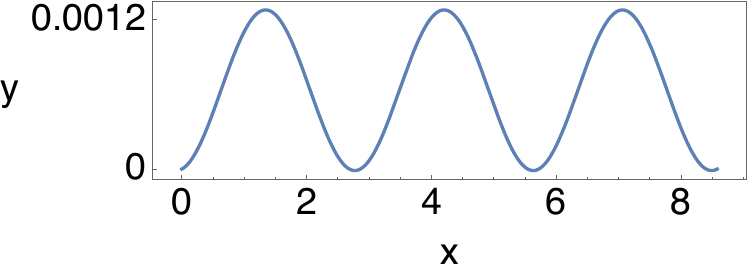}&
		\includegraphics[height=0.19\linewidth]{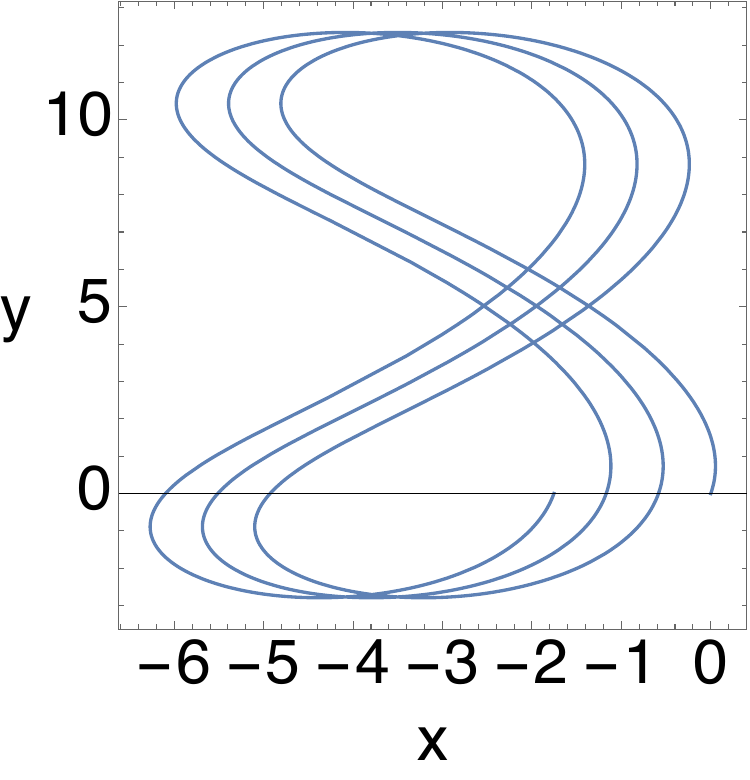} & 
		\includegraphics[height=0.18\linewidth]{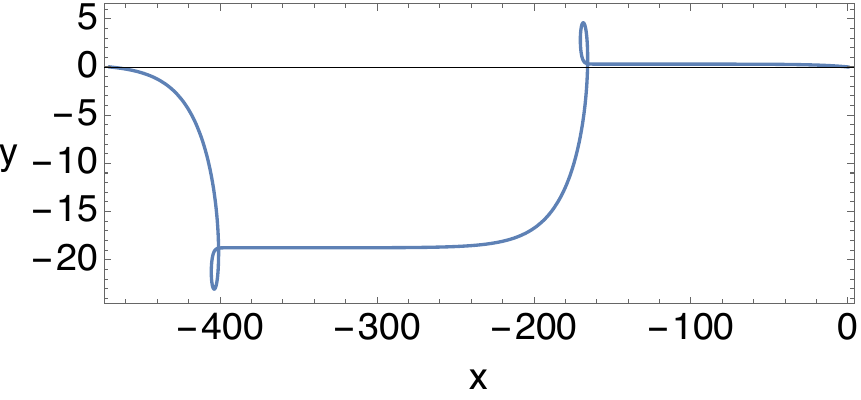}\\
		(a) & (b) & (c) 
	\end{tabular}
	\caption{Earth-fixed coordinate plots of the solutions from Fig.~\ref{periodic_orbit_etas} (a,b,c), respectively. 
	}
	\label{periodic_orbit_earth2}
\end{figure}

For further illustration, we plot the steering angle $\eta= \eps_\w\w+\eps_\psi\sin(\psi)$ of selected periodic solutions in Fig.~\ref{periodic_orbit_etas}. As expected, for a solution close to the Hopf bifurcation the steering angle just mildly oscillates. 
The variations are stronger for solutions that are further from it, but only up to angles of around $0.5$ \si{rad} $ \approx 28.6^\circ$; cf.\ Fig.~\ref{periodic_orbit_etas} (b,c).
	
\medskip
In conclusion, it appears that the bifurcating periodic solutions are confined to the region in parameter space in which the straight motion is unstable. In this region, we find either a periodic orbit with winding number zero, a symmetric pair of periodic orbits with winding number $1$ and $-1$, or a heteroclinic cycle. We expect that this global arrangement of periodic orbits in parameter space persists for periodic function in \eqref{e:sinus} that are perturbations of the sinus. However, typically there will be quantitative changes, e.g., the additional equilibria will not be exactly opposite to the reference direction $\psi=0$. 

\subsection{Periodic Solutions in Earth-Fixed Coordinates}
We {\black determine} the resulting ship {\black motions in the} Earth-fixed position coordinates $(x,y)\in\R^2$. These can be conveniently expressed in complex form as $z=x+ iy\in\C$ and the relation of $z(t)$ {\black to a ship-fixed trajectory $(u,v,\psi)(t)$} is given by $\dot z=(u+iv)\exp(i\psi)$. 
{\black An equilibrium point with constant $(u,v,\psi)$ corresponds to the straight motion of the ship if $r=\dot \psi=0$, and to a circular motion if $r=\dot \psi\neq 0$. Nevertheless, for periodic solutions $(u,v,\psi)(t)$, the resulting Earth-fixed positions do not need to be periodic.} 

{\black We plot tracks of Earth-fixed positions for selected solutions from Fig.~\ref{periodic_orbits_epsr10} in Fig.~\ref{periodic_orbit_earth2}.} 
{\black The result for a periodic solution near the Hopf bifurcation is shown in Fig.~\ref{periodic_orbit_earth2} (a). 
	The equilibrium point involved in the Hopf bifurcation is a straight motion along the $x$-axis, and the track of the bifurcated periodic solution slightly oscillates about this. The continuation further away from the bifurcation point, e.g., for $\eps_\psi=1$, results in tracks that oscillate with larger amplitude, but continue to monotonically drift along the $x$-axis. 
	However, decreasing $\eps_\psi$ further, the track turns into a near figure eight shape with self-intersections as shown in Fig.~\ref{periodic_orbit_earth2} (b); here $\eps_\psi\approx 0.419$. The $x$-component is non-monotone, but on average there is a drift along the $x$-axis. 
	Lastly, in Fig.~\ref{periodic_orbit_earth2} (c) we plot the track along one period of the periodic solution that is close to the heteroclinic cycle mentioned above for $\eps_\psi\approx 0.408$. 
	The track consists of two phases that each resemble one of the heteroclinic orbits involved in the heteroclinic cycle. These heteroclinic orbits are related by reflection and the track also appears to be symmetric about the mid-point between the global $y$-extrema. 
	The equilibria of the heteroclinic cycle correspond to straight motion opposite to the reference direction. 
	Along the track starting from the right boundary, the ship moves nearly straight along the negative $x$-axis, and then it makes a clockwise full turn that includes a drift to negative $y$-values. It continues roughly parallel along the negative $x$-axis and then turns counter-clockwise with a drift back to near $y=0$. It appears that this motion repeats periodically. 
	
	\medskip
	In conclusion, choosing control gains outside the stability boundary, but not far from it, the resulting ship track still essentially follows the straight heading. For a yaw restoring gain which is not `too small', the ship tracks are also qualitatively unaffected by the alteration of the control law. However, for small $\eps_\psi>0$ the choice of modification enters into account and, as plotted in Fig.~\ref{periodic_orbit_earth2} (c), a track occurs which strongly depends on the sinusoidal choice, although the steering angle does not exceed $\pm29^\circ$.
}

\section{Discussion}\label{Discussion}
{\black 
	We have presented a new method to determine the criticality of Hopf bifurcations for a class of models for ship maneuvering whose nonlinearities contain terms that are continuous but not smooth. 
	This method replaces the smooth theory of center manifolds and normal forms, which is in general not applicable for bifurcations from straight ship motion. 
	As usual in the study of bifurcations, a good understanding of the eigenvalues of the linearization in the underlying steady state is required. For the selected class of models we have therefore performed a detailed analysis of gain margins at which the steady motion changes stability. Here we assumed a standard proportional control for the steering angle that consists of a combination of yaw damping and yaw restoring control. In this study we have been able to identify the geometry of this stability boundary in control parameter space in terms of the ship characteristics. We have shown that the propeller diameter has no qualitative impact, but the position $x_T$ of the propulsion force does. For the HTC characteristics we found that the stable region changes shape at specific locations upon moving $x_T$ further to the fore. More generally, four types of stable regions can occur; in one case the stable region is bounded and otherwise it is unbounded. Stabilization by the chosen control becomes impossible when $x_T$ is larger than an explicit threshold, which is either $x_{T_1}$ or $\xTs<x_{T_1}$ from \eqref{e:xTthresholds}. Here $x_{T_1}=N_\beta/Y_\beta$ is a particularly simple ratio of non-dimensional hydrodynamic hull coefficients. In all cases, if the straight motion can be stabilized at all, then this is already possible for zero yaw restoring gain, i.e., $\eps_\psi=0$.
	
	Our computations of the first Lyapunov coefficient 
	showed that the resulting Andronov--Hopf bifurcations are always supercritical. Also a pitchfork bifurcation that occurs at $\eps_\psi=0$ is supercritical. Checking this criticality satisfies a request for a nonlinear assessment of stabilizing control from \cite{PAPOULIAS1994}. 
	The supercriticality means that the linear stability boundary can be viewed as a safe prediction for the stabilization of the straight motion equilibrium. 
	As a refinement of this we have simultaneously computed the sensitivity of the amplitude of the bifurcating solutions. We have found that it is more sensitive to the yaw damping gain for small yaw restoring gain, except in the case that the stable region is bounded and the yaw damping is relatively large.
	
	In order to further corroborate that the bifurcating solutions do not interfere with the stabilized course, we have studied the bifurcations more globally using numerical continuation. We found that the bifurcating periodic orbits are indeed fully confined to the region in parameter space where the straight motion is unstable. For the purpose of understanding this organization of solution branches near zero yaw restoring gain, we were forced to modify the proportional control to be periodic in the yaw angle. This modification is not canonical and has an impact on the global behavior of the system. For the specific choice \eqref{e:sinus} we have presented some of the resulting ship tracks. Those for $\eps_\w\approx 0$ are strongly affected by this modification, but those for, e.g., $\eps_\w\geq1$ seemed to be qualitatively unaffected. 
	
	Towards completing this theoretical study, it would be desirable to unfold the arising double zero eigenvalue point. Challenges are to account for the cylindrical geometry, the non-generic character of the model at $\eps_\psi=0$, and the nonsmoothness of nonlinear terms. 
	
	\medskip
	Concerning the safety of stabilizing the straight motion, it would be interesting to check the criticality of bifurcations in other regimes of ship design parameters, or to change the underlying model. An option would be to consider the more complex rudder model from \cite{ToxopeusThesis}, which we simplified to the `thruster' model \eqref{thruster_forces}, or the Ro-Pax ship characteristics and/or the $4$ degree-of-freedom model from \cite{STC2007}. 
	
	\medskip
	Finally, an extension in a different direction would be to move beyond the ship-fixed equilibrium maneuvers and employ numerical continuation to effectively investigate different planned movements. As an example we have the zig-zag maneuvers \cite{ToxopeusPaper}, and the evaluation of so-called overshoots \cite{StandShips2006}. 
}

\appendix
\section{HTC Characteristics}
\label{HTC_param}

{\black The parameter values listed stem from \cite{ToxopeusThesis}. A part of it can already be found in \cite{ToxopeusPaper}.}

\begin{table}[H]
	\centering
	\begin{tabular}{ | p{1.0cm} | c || p{1.0cm} | c || p{1.0cm} | c || p{1.0cm} | c |  }
		\hline
		\multicolumn{8}{|c|}{ \textbf{Hull forces} } \\
		\hline
		\textbf{Coeff.} & \textbf{Value} & \textbf{Coeff.} & \textbf{Value} & \textbf{Coeff.} & \textbf{Value} & \textbf{Coeff.} & \textbf{Value}\\
		\hline
		$m$ 				& $0.2328$  & $m_{uu}$ & $0.0247$ & $X_{u|u|}$ 			& $-0.0141$ & $N_\beta$				& $-0.1442$ \\
		${I}_z$ 			& $0.0134$  & $m_{vv}$ & $0.2286$ & $Y_\beta$ 			& $-0.1735$ & $N_\gamma$ 			& $-0.0276$ \\
		& 			& $m_{rr}$ & $0.0150$ & $Y_\gamma$ 			& $0.0338$ 	& $N_{\beta|\beta|}$		& $-0.0375$ \\
		& 			& $m_{vr}$ & $0.0074$ & $Y_{\beta|\beta|}$  & $-1.1378$ & $N_{\gamma|\gamma|}$	& $-0.0386$ \\
		&			& $m_{rv}$ & $0.0074$ & $Y_{\gamma|\gamma|}$& $0.0123$  & 	&  \\
		& & & & $Y_{\beta|\gamma|}$	& $-0.0537$ & 					    & \\
		& & & & $Y_{|\beta|\gamma}$	& $0.1251$  & 						& \\
		\hline
	\end{tabular}
	\caption{Rescaled added mass coefficients, \eqref{4DSystem}, and hydrodynamic bare hull coefficients, \eqref{thruster_forces}.}
	\label{Mass_coeffs}
\end{table}

\begin{table}[!htbp]
	\centering
	\begin{tabular}{ | p{1.0cm} | c || p{1.0cm} | c || p{1.0cm} | c | }
		\hline
		\multicolumn{4}{|c|}{ \textbf{Propeller characteristics} } & \multicolumn{2}{c|}{  }\\
		\hline
		\textbf{Coeff.} & \textbf{Value} & \textbf{Coeff.} & \textbf{Value} & \textbf{Coeff.} & \textbf{Value}\\
		\hline
		&   		& $K_{T0}$ & $0.366897$ & $\L$ 			& $153.70$\\
		$\bar{D}_p$ &  $6.105$ 	& $K_{T1}$ & $-0.345036$ 	& $T$	& $10.30$\\
		&			& $K_{T2}$ & $0.068841$ & $t$ & $0.22$ \\
		& 			& $K_{T3}$ & $-0.710991$ & $w$ & $0.38$ \\
		&			& $K_{T4}$ & $0.948559$ & & \\
		&			& $K_{T5}$ & $-0.428915$ & & \\
		\hline
	\end{tabular}
	\caption{Propeller characteristics with propeller diameter $[\bar{D}_p]=m$, non-dimensional coefficients $K_{Ti}$ of the propeller thrust $T_p$, corresponding to the model No. 5286., and further parameters: length between perpendiculars $[\L]=m$, mean draft $[T]=m$, and non-dimensional thrust deduction fraction $t$ and wake fraction $w$.}
	\label{Prop_char}
\end{table}

\section*{Acknowledgments}
This research has been supported by the Deutsche Forschungsgemeinschaft (DFG, German Research Foundation), 
project 281474342/GRK2224/1.
We thank Ed van Daalen from MARIN for initiating this study and sharing reports as well as supporting the initial implementation in \textsc{Auto} together with the intern Antoine Anceau. We also thank Mathias Temmen for contributing to the initial implementation and linear stability study 
in the context of his Msc thesis. We are grateful to Thor I. Fossen, from the NTNU, for helpful discussions.

\bibliographystyle{siamplain}
{\small \bibliography{arXiv_Steinherr_Rademacher_P2}}

\end{document}